\documentclass{amsart}
\usepackage[margin=1.4in]{geometry}
\usepackage{amssymb}
\usepackage{amsmath}
\usepackage{amsthm}
\usepackage{amsfonts}
\usepackage{graphicx}
\usepackage[all,cmtip]{xy}

\theoremstyle{plain}
\newtheorem{theorem}{Theorem}

\newtheorem{lemma}[theorem]{Lemma}

\newtheorem{proposition}[theorem]{Proposition}
\newtheorem{definition}[theorem]{Definition}

\theoremstyle{definition}
\newtheorem*{notation}{Notation}
\newtheorem*{acknowledgements}{Acknowledgements}
\newtheorem{example}[theorem]{Example}

\newtheorem{remark}{Remark}

\theoremstyle{remark}

\numberwithin{equation}{section}
\input{tcilatex}

\begin{document}

\title[Two-dimensional id\`{e}les]{Two-dimensional id\`{e}les with
cycle module coefficients}
\author[O. Braunling]{Oliver Braunling}
\address{Fakult\"{a}t f\"{u}r Mathematik, University of Duisburg-Essen,
Thea-Leymann-Stra\ss e 9, 45127 Essen, Germany\newline
oliver.braeunling@uni-due.de}
\thanks{This work has been partially supported by the DFG SFB/TR45
\textquotedblleft {Periods, moduli spaces, and arithmetic of algebraic
varieties}\textquotedblright and the Alexander von Humboldt Foundation.}

\begin{abstract}
We give a theory of id\`{e}les with coefficients for smooth surfaces over a
field. It is an analogue of Beilinson/Huber's theory of higher ad\`{e}les,
but handling cycle module sheaves instead of quasi-coherent ones. We prove
that they give a flasque resolution of the cycle module sheaves in the
Zariski topology. As a technical ingredient we show the Gersten property for
cycle modules on equicharacteristic complete regular local rings, which
might be of independent interest.
\end{abstract}

\maketitle

\section{\label{section_Motivation}Motivation}

Let $K$ be a number field, $\mathbf{A}_{K}$ its ad\`{e}les and $\mathbf{I}%
_{K}$ its id\`{e}les. Then the key invariants, the group of units and the
ideal class group are given by%
\begin{equation}
\mathcal{O}_{K}^{\times }=\mathbf{I}_{K}^{0}\cap K^{\times }\qquad \text{and}%
\qquad \limfunc{Cl}\nolimits_{K}=\mathbf{I}_{K}/(\mathbf{I}_{K}^{0}\cdot
K^{\times })\text{,}  \label{lX_1}
\end{equation}%
where $\mathbf{I}_{K}^{0}$ denotes the integral id\`{e}les. Many other
invariants can also be expressed through ad\`{e}les and id\`{e}les (e.g. $L$%
-function, discriminant, etc.), but let us stick to the above two for the
moment. Using the number field/function field dictionary, these facts
readily carry over to smooth curves over finite fields, and in fact to
smooth curves $X/k$ over arbitrary fields. Viewing $X$ as a scheme, we may
define sheaf versions, namely%
\begin{equation}
\mathbf{I}^{0}(U):=\tprod\nolimits_{x\in U^{(1)}}\widehat{\mathcal{O}}%
_{x}^{\times }\qquad \text{and}\qquad \mathbf{I}(U):=\left.
\tprod\nolimits_{x\in U^{(1)}}^{\prime }\right. \widehat{K}_{x}^{\times }%
\text{,}  \label{lX_2}
\end{equation}%
where $U$ is a Zariski open, $U^{(1)}$ the set of closed points in $U$, $%
\widehat{\mathcal{O}}_{x}$ the completed local ring at $x$, $\widehat{K}_{x}$
its field of fractions and the prime superscript in $\prod^{\prime }$ means
that we restrict to elements such that all but finitely many components lie
in the subgroup $\widehat{\mathcal{O}}_{x}^{\times }\subseteq \widehat{K}%
_{x}^{\times }$. The $\widehat{K}_{x}$ are local fields (in this text a 
\textit{local field} refers to a complete discrete valuation field); here
all of the shape of Laurent series fields%
\begin{equation}
\kappa ((t))\text{,}\qquad \text{while}\qquad \mathbf{Q}_{p},\quad \mathbf{R}
\label{lX_4}
\end{equation}%
and their finite extensions are the local fields appearing in the classical
theory for number fields. One obtains an exact sequence of sheaves of
abelian groups,%
\begin{equation}
0\longrightarrow \mathcal{O}_{X}^{\times }\overset{\limfunc{diag}}{%
\longrightarrow }\mathbf{I}^{0}\oplus \mathcal{K}^{\times }\overset{\limfunc{%
diff}}{\longrightarrow }\mathbf{I}\longrightarrow 0\text{,}  \label{lX_3}
\end{equation}%
where $\mathcal{K}$ denotes the sheaf of rational functions. In fact this is
a flasque resolution of $\mathcal{O}_{X}^{\times }$ and thus%
\begin{equation*}
H^{0}(X,\mathcal{O}_{X}^{\times })=\mathbf{I}^{0}(X)\cap K^{\times }\qquad 
\text{and}\qquad H^{1}(X,\mathcal{O}_{X}^{\times })=\mathbf{I}(X)/(\mathbf{I}%
^{0}(X)\cdot K^{\times })\text{.}
\end{equation*}%
In view of $H^{1}(X,\mathcal{O}_{X}^{\times })=\limfunc{Pic}X\cong \limfunc{%
Cl}X$ we have recovered a perfect analogue of eq. \ref{lX_1}. Although it
might seem unnecessarily fancy, let us write $\limfunc{CH}^{1}(X)$ instead
of $\limfunc{Cl}(X)$ from now on. Without needing to change definitions, the
above also works for smooth surfaces (indeed in any dimension). That should
make us curious, since for smooth surfaces we have the intersection pairing%
\begin{equation}
\xymatrix{ \limfunc{CH}\nolimits^{1}(X) \otimes _{\mathbf{Z}}
\limfunc{CH}\nolimits^{1}(X) \ar[r] \ar[d] & \limfunc{CH}\nolimits^{2}(X)
\ar[d]^-{b_?} & \text{\textquotedblleft Chow side\textquotedblright } \\
{\frac{\mathbf{I}(X)}{\mathbf{I}^{0}(X)\cdot K^{\times }}} \otimes
_{\mathbf{Z}} {\frac{\mathbf{I}(X)}{\mathbf{I}^{0}(X)\cdot K^{\times }}}
\ar[r]_-{a_?} &{?} & \text{\textquotedblleft idelic side\textquotedblright }
}  \label{lX_5}
\end{equation}%
It is now very natural to ask:

\begin{enumerate}
\item Is there also an id\`{e}le type\ object which can be placed in the
lower-right corner?

\item What cohomology groups besides $H^{i}(X,\mathcal{O}_{X}^{\times })$, $%
\limfunc{CH}^{i}(X)$ admit such an \textquotedblleft id\`{e}le
type\textquotedblright\ counterpart?
\end{enumerate}

These questions rest on the idea that there should exist general \textit{id%
\`{e}les with coefficients} which provide flasque resolutions for sheaves,
thus \textit{expressing cohomology groups through subquotients of suitable id%
\`{e}les}.\medskip \newline
We prove:

\begin{theorem}
\label{INTRO_MAINTHM}(Main Theorem) Let $X/k$ be a smooth integral surface
over an arbitrary field $k$, $M_{\ast }$ a big cycle module (see \S \ref%
{section_CycleModules} for the definition) and $\mathcal{M}_{\ast }$ the
associated Zariski sheaf. Then id\`{e}le sheaves with coefficients in $%
M_{\ast }$ (defined in \S \ref{section_IdelesNoProduct}) give a flasque
resolution%
\begin{equation*}
0\rightarrow \mathcal{M}_{\ast }\rightarrow \mathbf{I}_{\mathcal{M}%
}^{0}\longrightarrow \mathbf{I}_{\mathcal{M}}^{1}\longrightarrow \mathbf{I}_{%
\mathcal{M}}^{2}\longrightarrow 0\text{.}
\end{equation*}
\end{theorem}

As a tool for the proof of the main theorem, we need the following version
of the Gersten conjecture for complete local rings:

\begin{theorem}
Suppose $X/k$ is an equicharacteristic complete regular local scheme and let 
$M_{\ast }$ be a big cycle module (see \S \ref{section_CycleModules} for the
definition). Then the Rost cycle complex $C^{\bullet }\left( X,M_{\ast
}\right) $ is exact.
\end{theorem}

See \S \ref{section_AppendixCompleteGersten} for this result, which can be
read independently of the rest of the text. While one would certainly expect
its truth, we have been unable to find a literature reference.\medskip 
\newline
\textit{An application:\medskip }

Once established, this mechanism bridges freely between the cohomological
and the id\`{e}le perspective. For example, suppose $X/\mathbf{F}_{q}$ is a
smooth proper geometrically integral surface, now over a finite field.
Classical work of Kato and Saito \cite{MR717824} shows that there is an
unramified class field theory, a morphism%
\begin{equation}
\limfunc{rec}\nolimits_{X}:\limfunc{CH}\nolimits_{0}(X)\hookrightarrow \pi
_{1}^{\text{\'{e}t}}(X,\overline{x})^{\limfunc{ab}}\text{,}  \label{lX_100}
\end{equation}%
where `$\limfunc{ab}$' denotes abelianization. It is injective with dense
image. It sends a closed point to its Frobenius, just like in classical
class field theory for curves. Since $\limfunc{CH}\nolimits_{0}(X)$ appears
as a cohomology group of the $K$-theory sheaf $\mathcal{K}_{2}$, the above
mechanism automatically produces an id\`{e}le counterpart, which can be
phrased as%
\begin{equation}
\limfunc{rec}\nolimits_{X,\text{id\`{e}lic}}:\frac{\prod_{x,y_{i}^{\prime
}}^{\prime \prime \prime }K_{2}(\widehat{K}_{x,y_{i}^{\prime }})}{%
\prod_{x}K_{2}(\widehat{K}_{x})+\prod_{y}K_{2}(\widehat{K}%
_{y})+\prod_{x,y_{i}^{\prime }}K_{2}(\widehat{\mathcal{O}}_{x,y_{i}^{\prime
}})}\rightarrow \pi _{1}^{\text{\'{e}t}}(X,\overline{x})^{\limfunc{ab}}\text{%
.}  \label{lX_101}
\end{equation}%
Here $K_{n}$ denotes either ordinary $K$-theory groups or Fesenko's
topological Milnor $K$-groups; this does not affect the quotient. The key
point is that the local fields $\mathbf{Q}_{p}$, $\kappa ((t))$ constituting
the classical id\`{e}les are being replaced by $2$-local fields $\widehat{K}%
_{x,y_{i}^{\prime }}\approx \kappa ((s))((t))$ depending on flags of points $%
x$ lying on curves $y$.\medskip

We do \textit{not} intend to reprove class field theoretic results. The
point of the paper is that the equivalence of the two seemingly very
different formulations (eq. \ref{lX_100} versus eq. \ref{lX_101}) follows
immediately from an id\`{e}le resolution of the $\mathcal{K}_{2}$-sheaf. In
the same way one can complete Fig. \ref{lX_5}. See \cite{MR565095}, \cite%
{MR1105583}, \cite{MR1138291}, \cite{MR1478630}, \cite{MR2354210}, \cite%
{MR2489487} for other ad\`{e}le/id\`{e}le theories resolving particular
types of sheaves under varying assumptions.\medskip

\textit{History:} Chevalley initiated the theory of id\`{e}les in 1936 for
class field theory. The theory then was developed further by Weil, Tate and
Iwasawa; finiteness theorems and $L$-functions entered the picture. Parshin
was the first to formulate a theory of ad\`{e}les for surfaces \cite%
{MR0419458}, but already had a far broader range of applications in mind
than just class field theory (mostly unpublished \cite{MR697316}, \cite%
{ParshinAdeleTheory}). Beilinson then gave a general theory of ad\`{e}les in
all dimension with quasi-coherent coefficients \cite{MR565095}, with
detailed proofs written down by Huber \cite{MR1105583}, \cite{MR1138291}.
These theories are all `finite type over a field'. Fesenko developed a
theory `over $\mathbf{Z}$', also incorporating infinite places, versions of
harmonic analysis and $L$-functions \cite{MR2046602}, \cite{MR2462437}, \cite%
{MR2658047}. Osipov gave a theory for $K$-theory coefficients \cite%
{MR1478630}. Gorchinskiy developed a theory of `not adically completed id%
\`{e}les' with very general coefficients and in all dimensions \cite%
{MR2354210}, \cite{MR2489487}. A similar
`adically uncompleted' theory for quasi-coherent coefficients is due to
Huber, again \cite{MR1138291}.

\section{\label{section_CycleModules}Cycle modules}

\subsection{Axioms for big range of definition}

We will develop a flasque resolution in terms of \textquotedblleft id\`{e}le
sheaves\textquotedblright\ for sheaves which come from Rost cycle modules
(introduced in \cite{MR1418952}), generalizing the mechanism explained in \S %
\ref{section_Motivation} for curves. Actually, we \textit{need} to work with
a broader range of definition than in the original paper, so it is worth
stating the axioms we use in their precise form:\newline
Fix a base field $k$. For every field extension $F/k$ let $K_{\ast
}^{M}\left( F\right) $ denote the $\mathbf{Z}_{\geq 0}$-graded
anticommutative Milnor $K$-theory ring over $F$,%
\begin{equation*}
K_{\ast }^{M}\left( F\right) =\tcoprod\nolimits_{n\geq 0}K_{n}^{M}\left(
F\right) \qquad \qquad =\mathbf{Z}\oplus F^{\times }\oplus K_{2}^{M}\left(
F\right) \oplus \ldots \text{.}
\end{equation*}%
Then a \emph{big cycle module} is an object function%
\begin{equation}
M_{\ast }:\{\text{field extensions }F/k\}\longrightarrow \{\mathbf{Z}\text{%
-graded abelian groups}\}  \label{lX_6}
\end{equation}%
along with some extra structure \textbf{D1}-\textbf{D4} satisfying various
axioms. The extra structure (called \emph{transfers}) is given as follows:
We shall tacitly assume that all our fields are extensions of $k$.

\begin{description}
\item[D1] For every field extension $F\hookrightarrow E$ there is an abelian
group morphism of relative degree $0$, $\limfunc{res}\nolimits_{F}^{E}:M_{%
\ast }\left( F\right) \rightarrow M_{\ast }\left( E\right) $, called \emph{%
restriction}.

\item[D2] For every \textit{finite} field extension $F\hookrightarrow E$
there is an abelian group morphism of relative degree $0$, $\limfunc{cor}%
\nolimits_{F}^{E}:M_{\ast }\left( E\right) \rightarrow M_{\ast }\left(
F\right) $, called \emph{corestriction} (or \emph{norm}).

\item[D3] For every field $F$ the group $M_{\ast }\left( F\right) $ has the
structure of a graded left-$K_{\ast }^{M}\left( F\right) $-module,%
\begin{equation*}
K_{n}^{M}\left( F\right) \cdot M_{m}\left( F\right) \longrightarrow
M_{n+m}\left( F\right) \text{.}
\end{equation*}

\item[D4] For every field $F$ and every discrete excellent $k$-trivial
valuation $v$ on $F$ (see Rmk. \ref{marker_valuations} for terminology),
there is an abelian group homomorphism of relative degree $-1$,%
\begin{equation*}
\partial _{v}^{F}:M_{\ast }\left( F\right) \longrightarrow M_{\ast -1}\left(
\kappa \left( v\right) \right) \text{,}
\end{equation*}%
where $\kappa \left( v\right) $ denotes the residue field of the valuation
ring $A_{v}\subseteq F$ corresponding to $v$. Then $\partial _{v}^{F}$ is
called the \emph{boundary map} at $v$.
\end{description}

\begin{remark}
\label{marker_valuations}(Valuations) \emph{Discrete} means that the value
group is isomorphic to $\mathbf{Z}$. This implies that the ring of integers
(= valuation ring) is Noetherian. \emph{Excellent} means that we demand the
ring of integers to be excellent. Finally, $k$\emph{-trivial} means that the
restriction of the valuation to $k$ is trivial. Without this condition the
residue field $\kappa (v)$ would not need to be an extension of $k$. Also,
this excludes for example $p$-adic valuations.
\end{remark}

\begin{remark}
(Changes in the definition) This differs from the original definition in 
\cite{MR1418952} as follows: Rost demands $M_{\ast }$ in eq. \ref{lX_6} only
to be defined and have transfers on \emph{finitely generated} field
extensions of $k$. However, this excludes many interesting maps, e.g.%
\begin{equation*}
\limfunc{res}\nolimits_{k(t)}^{k((t))}:M_{\ast }(k(t))\rightarrow M_{\ast
}(k((t)))\qquad \limfunc{res}\nolimits_{\mathbf{Q}}^{\mathbf{Q}_{p}}:M_{\ast
}(\mathbf{Q})\rightarrow M_{\ast }(\mathbf{Q}_{p})
\end{equation*}%
coming from completions. Such maps abound in the construction of id\`{e}les.
They play a key role in arithmetic geometry, for example in the definition
of Selmer and Tate-Shafarevich groups or the Brauer-Manin obstruction. Next,
in \textbf{D4} Rost restricts to valuations such that the transcendence
degree over $k$ from $F$ to $\kappa (v)$ drops by one (\textquotedblleft
geometric valuations\textquotedblright ). Again, this is too restrictive for
our purposes, e.g. the $t$-valuation of $k((t))$ has residue field $k$ and
thus is \emph{not} geometric. The relaxed condition adds many valuations,
even to finitely generated fields. For example, suppose $k=\mathbf{Q}$ and
take $F:=\mathbf{Q}(s,t)$. Restrict the $s$-valuation coming from $%
F\hookrightarrow \mathbf{Q}((s))$, $t\mapsto se^{s}$, to $F$. It has ring of
integers and residue field%
\begin{equation*}
A_{v}:=\mathbf{Q}(s,se^{s})\cap \mathbf{Q}[[s]]\qquad \qquad \kappa (v)=%
\mathbf{Q}\text{,}
\end{equation*}%
so a drop of transcendence degree by two. The ring $A_{v}$ is excellent \cite%
[Thm. 1, Folgerung 3 (i)]{MR0427319}. This valuation would not be allowed in
Rost's original theory.\footnote{%
Rost already points out in his paper that the present relaxed axioms are
feasible, see for example \cite[Rmk. 1.8, Rmk. 2.8]{MR1418952}.}
\end{remark}

The transfer morphisms from \textbf{D1}-\textbf{D4} need to satisfy a long
list of axioms, \textbf{R1a}-\textbf{R3e }(\textquotedblleft cycle premodule
axioms\textquotedblright ) and the two crucial axioms (\textquotedblleft
cycle module axioms\textquotedblright\ \textbf{FD}, \textbf{C}); they can be
taken over literally from \cite{MR1418952} (except that field extensions
need not be finitely generated).

\begin{example}
Milnor $K$-theory, Quillen $K$-theory and Galois cohomology groups with ($%
\neq \limfunc{char}k$) torsion coefficients (in the sense of \cite[Rmk. 1.11]%
{MR1418952}) are big cycle modules. See \cite[Thm. 1.4 and following
paragraph]{MR1418952}. In principle, one can formulate categories of cycle
modules%
\begin{equation*}
CM_{\limfunc{big}}\rightleftarrows CM
\end{equation*}%
with $CM$ the theory as in Rost's work, and $CM_{\limfunc{big}}$ with the
relaxed assumptions here. There is an obvious restriction functor to the
right and one can use the left-adjoint to \textquotedblleft
prolong\textquotedblright\ classical cycle modules to the present setup. I
thank F. D\'{e}glise and S. Kelly for explaining this to me.
\end{example}

We shall now repeat some particularly important constructions from \cite%
{MR1418952}, even though no changes are neccessary, but we will often refer
to them:\medskip \newline
Let $X$ be an excellent scheme. We shall write $X^{(i)}$ for the set of
codimension $i$ points. Suppose $X\rightarrow \limfunc{Spec}k$, $x\in
X^{\left( i\right) }$ for some $i$, $y\in X^{\left( i+1\right) }$, $y\in 
\overline{\{x\}}$, i.e. $y$ can be viewed as a codimension one point in the
integral closed subscheme $\overline{\{x\}}\hookrightarrow X$. As in \cite%
{MR1418952} define a \emph{boundary map}%
\begin{equation*}
\partial _{y}^{x}:M_{\ast }\left( \kappa \left( x\right) \right)
\longrightarrow M_{\ast -1}\left( \kappa \left( y\right) \right)
\end{equation*}%
as follows: The stalk $\mathcal{O}_{\overline{\{x\}},y}$ is a
one-dimensional local domain. By excellence, the normalization%
\begin{equation}
\varphi :\limfunc{Spec}\mathcal{O}_{\overline{\{x\}},y}^{\prime
}\longrightarrow \limfunc{Spec}\mathcal{O}_{\overline{\{x\}},y}  \label{l317}
\end{equation}%
inside its own field of fractions $\kappa \left( x\right) $ is a finite
morphism, so the closed point $y$ in $\limfunc{Spec}\mathcal{O}_{\overline{%
\{x\}},y}$ has finite preimage $\{y_{1}^{\prime },\ldots ,y_{r}^{\prime
}\}\subset \limfunc{Spec}\mathcal{O}_{\overline{\{x\}},y}^{\prime }$. By
normality the localizations $(\mathcal{O}_{\overline{\{x\}},y}^{\prime
})_{y_{i}^{\prime }}$ are DVRs. Let $v_{1},\ldots ,v_{r}$ denote the
corresponding discrete valuations of the fraction field $\kappa \left(
x\right) $. Now one defines%
\begin{equation}
\partial _{y}^{x}:=\tsum\nolimits_{i=1}^{r}\limfunc{cor}\nolimits_{\kappa
\left( y\right) }^{\kappa \left( v_{i}\right) }\circ \partial
_{v_{i}}^{\kappa \left( x\right) }:M_{\ast }\left( \kappa \left( x\right)
\right) \longrightarrow M_{\ast -1}\left( \kappa \left( y\right) \right) 
\text{,}  \label{l318}
\end{equation}%
where $\kappa \left( v_{i}\right) $ denotes the residue field with respect
to the valuation $v_{i}$ (i.e. the residue field $\kappa \left(
y_{i}^{\prime }\right) $ of the ring $\mathcal{O}_{\overline{\{x\}}%
,y}^{\prime }$), $\partial _{v_{i}}^{\kappa \left( x\right) }$ is as in 
\textbf{D4}. Note that the finiteness of the normalization morphism $\varphi 
$, eq. \ref{l317}, implies that $\kappa \left( v_{i}\right) /\kappa \left(
y\right) $ is a \textit{finite} field extension, so \textbf{D2} is indeed
available.\medskip \newline
For every excellent $n$-equidimensional scheme $X\rightarrow \limfunc{Spec}k$
we define%
\begin{equation}
C^{p}\left( X,M_{q}\right) :=\coprod\nolimits_{x\in X^{\left( p\right)
}}M_{q-p}\left( \kappa \left( x\right) \right) \text{.\qquad (also denoted
by }C^{p}(X;M,q)\text{)}  \label{l446}
\end{equation}%
Here $q$ is a $\mathbf{Z}$-grading and along $p\geq 0$ one obtains a complex 
$C^{\bullet }(X,M_{q})$ by defining the $q$-degree-preserving differential%
\begin{equation*}
\partial _{X}:C^{p}\left( X,M_{q}\right) \longrightarrow C^{p+1}\left(
X,M_{q}\right)
\end{equation*}%
by using for all pairs $x\in X^{\left( p\right) }$ and $y\in X^{\left(
p+1\right) }$, $y\in \overline{\{x\}}$ (i.e. $y$ is a codimension one point
of the integral closed subscheme $\overline{\{x\}}$), the boundary map%
\begin{equation}
\partial _{y}^{x}:M_{\ast }\left( \kappa \left( x\right) \right)
\longrightarrow M_{\ast -1}\left( \kappa \left( y\right) \right)
\label{l334}
\end{equation}%
of eq. \ref{l318} and defining $\partial _{X}:=\sum \partial _{y}^{x}$ for
all such pairs $y\in \overline{\{x\}}$. This is a differential, $\partial
_{X}^{2}=0$, see \cite[Lemma 3.3]{MR1418952}. In fact this property is
essentially equivalent to axiom \textbf{C} of a cycle module. One defines%
\begin{equation}
M_{\ast }\left( X\right) :=\ker \left( \partial _{X}:C^{0}\left( X,M_{\ast
}\right) \rightarrow C^{1}\left( X,M_{\ast }\right) \right) \text{.}
\label{l322_DefOfUnramCycles}
\end{equation}%
Hence, $M_{\ast }\left( X\right) $ is a subgroup of $C^{0}\left( X,M_{\ast
}\right) $ and its elements are called \emph{unramified} cycles, see \cite[%
p. 338]{MR1418952}. Letting $U\mapsto M_{\ast }(U)$ for Zariski opens $U$,
we have a sheaf $\mathcal{M}_{\ast }\left( U\right) :=M_{\ast }\left(
U\right) $. Moreover, we call%
\begin{equation}
C^{\bullet }\left( X,M_{\ast }\right) :0\longrightarrow M_{\ast }\left(
X\right) \longrightarrow C^{0}\left( X,M_{\ast }\right) \overset{\partial
_{X}}{\longrightarrow }C^{1}\left( X,M_{\ast }\right) \overset{\partial _{X}}%
{\longrightarrow }\cdots  \label{l404}
\end{equation}%
the \emph{Rost cycle complex} of $X$ \emph{with coefficients in} $M_{\ast }$%
. Unlike \cite{MR1418952} we prefer to index by codimension rather than
dimension.

\subsection{\label{Subsect_FlatPullbackOfNonconstRelDim}Pullback \&\
Pushforward}

Suppose $X,Y$ are excellent schemes, say they are $n_{X}$-, $n_{Y}$-equidimensional, $%
N:=n_{Y}-n_{X}$. If $X,Y$ are finite type over a field and $f:X\rightarrow Y$
is proper, Rost constructs a pushforward%
\begin{equation*}
f_{\ast }:C^{\bullet }\left( X,M_{\ast }\right) \rightarrow C^{\bullet
+N}(Y,M_{\ast })\text{,}
\end{equation*}%
see \cite[\S 3.4]{MR1418952}. This construction carries over to the case of $%
X,Y$ just equidimensional excellent and $f$ a \textit{finite} morphism, see
Lemma \ref{Lemma_BasicFunctorialityPropsForCycleModules} below.\newline
If $X,Y$ are finite type over a field and $f:X\rightarrow Y$ is flat, Rost
constructs a pullback operation. To prolong this to excellent schemes, we
say that a \emph{flat morphism of constant relative dimension }$r$ is a flat
morphism $f:X\rightarrow Y$ such that there is some $r\in \mathbf{Z}$ such
that for all $y\in \limfunc{im}f$ and all generic points $x\in f^{-1}(%
\overline{\{y\}})^{\left( 0\right) }$ one has

\begin{enumerate}
\item $\dim_{X}x=\dim_{Y}y+r$.

\item or equivalently $\dim \mathcal{O}_{X,x}\otimes _{\mathcal{O}%
_{Y,y}}\kappa \left( y\right) =-N-r\geq 0$.
\end{enumerate}

See \cite[Ch. IX]{MR2427530} for this approach. Then there is a flat pullback%
\begin{equation*}
f^{\ast }:C^{\bullet }\left( Y,M_{\ast }\right) \rightarrow C^{\bullet
-r-N}(X,M_{\ast })\text{,}
\end{equation*}%
see \cite[\S 3.5]{MR1418952} (Rost also develops a more general pullback
operation, for which we have no use).

\begin{lemma}
\label{Lemma_BasicFunctorialityPropsForCycleModules}Suppose

\begin{itemize}
\item $X,Y,Z$ are $n_{X}$-, $n_{Y}$-, $n_{Z}$-equidimensional excellent
schemes, $N:=n_{Y}-n_{X}$,

\item $f:X\rightarrow Y$ finite,

\item $g:Z\rightarrow Y$ is flat of constant relative dimension $r$,

\item $W:=Z\times _{Y}X$,

\item $f^{\prime },g^{\prime }$ are the morphisms induced by base change.
\end{itemize}

Then the diagrams%
\begin{align*}
& \begin{xy} \xymatrix@!0@C=30mm@R=15mm{ C^{p}(X,M_{\ast})
\ar[r]^{\partial_{X}} \ar[d]_{f_{\ast}} & C^{p+1}(X,M_{\ast })
\ar[d]^{f_{\ast}} & C^{p}(X,M_{\ast}) \ar[r]^{g^{\prime\ast}}
\ar[d]_{f_{\ast}} & C^{p-r-N}(W,M_{\ast}) \ar[d]^{f^{\prime}_{\ast}} \\
C^{p+N}(Y,M_{\ast}) \ar[r]_{\partial_{Y}} & C^{p+N+1}(Y,M_{\ast}) &
C^{p+N}(Y,M_{\ast}) \ar[r]_{g^{\ast}} & C^{p-r}(Z,M_{\ast}) \\ } \end{xy} \\
& \begin{xy} \xymatrix@!0@C=40mm@R=15mm{ C^{p}(Y,M_{\ast})
\ar[r]^{\partial_{Y}} \ar[d]_{g^{\ast}} & C^{p+1}(Y,M_{\ast })
\ar[d]^{g^{\ast}} \\ C^{p-r-N}(Z,M_{\ast}) \ar[r]_{\partial_{Z}} &
C^{p-r-N+1}(Z,M_{\ast}) \\ } \end{xy}
\end{align*}%
commute. Moreover, $(f_{1}\circ f_{2})_{\ast }=f_{1\ast }\circ f_{2\ast }$
and $\left( g_{1}\circ g_{2}\right) ^{\ast }=g_{2}^{\ast }\circ g_{1}^{\ast
} $ for $f_{1},f_{2}$ two finite morphisms $X\overset{f_{2}}{\rightarrow }Y%
\overset{f_{1}}{\rightarrow }Y^{\prime }$, $g_{1},g_{2}$ two flat constant
relative dimension morphisms $Z\overset{g_{2}}{\rightarrow }Y\overset{g_{1}}{%
\rightarrow }Y^{\prime }$, $r\left( g_{1}\circ g_{2}\right)
=r(g_{1})+r(g_{2})$ for the constant relative dimensions respectively.
\end{lemma}

See \cite[Prop. 4.1, 4.6]{MR1418952}, the proofs carry over verbatim (but
note that our $f$ is finite as opposed to just being proper). See also \cite[%
Ch. IX]{MR2427530} for the case $M_{\ast }=K_{\ast }^{M}$. We will use the
above facts without further quoting this lemma.

\section{\label{section_IdelesWithQCCoeffs}Ad\`{e}les with quasi-coherent
coefficients}

Firstly, let us explain how to find the correct notion of multidimensional ad%
\`{e}les. The classical ad\`{e}les are built from local fields, which
usually can be written as an ind-pro limit, for example%
\begin{equation}
\mathbf{Q}_{p}=\underrightarrow{\limfunc{colim}}_{j}\underleftarrow{\lim }%
_{i}p^{-j}\mathbf{Z}/p^{i}\mathbf{Z}\qquad \kappa ((t))=\underrightarrow{%
\limfunc{colim}}_{j}\underleftarrow{\lim }_{i}t^{-j}k[t]/t^{i}\text{,}
\label{lX_7}
\end{equation}%
running over the system $i,j\rightarrow +\infty $. In the number field case,
the infinite places $\mathbf{R}$, $\mathbf{C}$ form an exception to this
principle; however, such places do not occur in the geometric id\`{e}le
theory of this text. One may now attempt to `globalize' such an ind-pro
perspective to a whole scheme. This leads to Beilinson's theory of ad\`{e}%
les for quasi-coherent sheaves \cite{MR565095}. We quickly summarize the key
points:\newline
Let $X$ be a Noetherian scheme. For scheme points $\eta _{0},\eta _{1}\in X$
write $\eta _{0}\geq \eta _{1}$ if $\overline{\{\eta _{0}\}}\ni \eta _{1}$
(\textquotedblleft specialization\textquotedblright ). Denote by $S\left(
X\right) _{n}:=\{(\eta _{0}\geq \cdots \geq \eta _{n}),\eta _{i}\in X\}$ the
set of chains of points of length $n+1$. Let $K_{n}\subseteq S\left(
X\right) _{n}$ be an arbitrary subset. For any point $\eta _{0}\in X$ define 
$\left. _{\eta _{0}}K\right. :=\{(\eta _{1}\geq \cdots \geq \eta _{n})$ s.t. 
$(\eta _{0}\geq \cdots \geq \eta _{n})\in K_{n}\}$, a subset of $S\left(
X\right) _{n-1}$. This can be viewed as the set of `subscheme flags' below $%
\eta $. Let $\mathcal{F}$ be a \textit{coherent} sheaf on $X$. For $n=0$ and 
$n\geq 1$ define respectively (and inductively)%
\begin{eqnarray}
A(K_{0},\mathcal{F}):= &&\prod\nolimits_{\eta \in K_{0}}\underleftarrow{\lim 
}_{i}\mathcal{F}\otimes _{\mathcal{O}_{X}}\mathcal{O}_{X,\eta }/\mathfrak{m}%
_{\eta }^{i}  \label{TATEMATRIX_l6} \\
A(K_{n},\mathcal{F}):= &&\prod\nolimits_{\eta \in X}\underleftarrow{\lim }%
_{i}A(\left. _{\eta }K_{n}\right. ,\mathcal{F}\otimes _{\mathcal{O}_{X}}%
\mathcal{O}_{X,\eta }/\mathfrak{m}_{\eta }^{i})\text{.}  \notag
\end{eqnarray}%
The sheaves $\mathcal{F}\otimes _{\mathcal{O}_{X}}\mathcal{O}_{X,\eta }/%
\mathfrak{m}_{\eta }^{i}$ are usually only quasi-coherent, so we complete
this definition as follows: For a \textit{quasi-coherent} sheaf $\mathcal{F}$
we define $A(K_{n},\mathcal{F}):=\underrightarrow{\limfunc{colim}}_{\mathcal{%
F}_{j}}A(K_{n},\mathcal{F}_{j})$, where $\mathcal{F}_{j}$ runs through all
coherent subsheaves of $\mathcal{F}$ (and hereby reducing to eqs. \ref%
{TATEMATRIX_l6}). Built successively from colimits and Mittag-Leffler
limits, $A(K_{n},-)$ is a covariant exact functor from quasi-coherent
sheaves to quasi-coherent sheaves.

\begin{example}
\label{Example_AutomaticFinitenessCondition}(Automatic finiteness condition)
For an integral smooth curve $X/k$ with function field $K$ this recovers the
classical ad\`{e}les for $\mathcal{F}:=\mathcal{O}_{X}$. Namely, for $%
\triangle :=S(X)_{1}$ one computes%
\begin{eqnarray}
A(\triangle ,\mathcal{O}_{X}) &=&A(\left. _{\eta }\triangle \right. ,K)=%
\underset{D}{\underrightarrow{\limfunc{colim}}}A(\left. _{\eta }\triangle
\right. ,\mathcal{O}_{X}(D))  \notag \\
&=&\underset{D}{\underrightarrow{\limfunc{colim}}}\prod\nolimits_{x\in
X^{(1)}}\underset{i}{\underleftarrow{\lim }}\underset{E\not\ni x}{%
\underrightarrow{\limfunc{colim}}}\mathcal{O}_{X}(D+E)/x^{i}  \notag \\
&=&\underset{D}{\underrightarrow{\limfunc{colim}}}\prod\nolimits_{x\in
X^{(1)}}\widehat{\mathcal{O}}_{x}\otimes \mathcal{O}_{x}(D)\qquad \subseteq
\prod\nolimits_{x\in X^{(1)}}\widehat{K}_{x}  \notag \\
&=&\left. \prod\nolimits_{x\in X^{(1)}}^{\prime }\right. \widehat{K}_{x}%
\text{,}  \label{lX_9}
\end{eqnarray}%
where $D$ runs through all divisors so that $\underrightarrow{\limfunc{colim}%
}_{D}\mathcal{O}_{X}(D)=K$ and $E$ through all divisors whose support does
not contain $x$ so that $\underrightarrow{\limfunc{colim}}_{E}\mathcal{O}%
_{X}(E)=\mathcal{O}_{x}$. The prime superscript in eq. \ref{lX_9} means that
for all but finitely many $x$ the component needs to lie in $\widehat{%
\mathcal{O}}_{x}\subseteq \widehat{K}_{x}$. The key point is: \emph{This
classical finiteness condition comes out automatically from Beilinson's
definition.}
\end{example}

Now, $S(X)_{\bullet }$ carries a natural structure of a simplicial set
(omitting the $i$-th entry in a flag yields faces; duplicating the $i$-th
entry in a flag degeneracies). This turns%
\begin{equation*}
\mathbf{A}^{\bullet }(U,\mathcal{F}):=A(S(U)_{\bullet },\mathcal{F})\qquad 
\text{(for }U\text{ Zariski open)}
\end{equation*}%
into a sheaf of cosimplicial abelian groups. We prefer to read it as a
complex of sheaves (via the unreduced Dold-Kan correspondence) with its
terms called $\mathbf{A}_{\mathcal{F}}^{i}$.

\begin{theorem}
\label{lX_BeilinsonResolutionThm}(Beilinson \cite[\S 2]{MR565095}) For a
Noetherian scheme $X$ and a quasi-coherent sheaf $\mathcal{F}$ on $X$, the $%
\mathbf{A}^{i}(-,\mathcal{F})$ are flasque sheaves and%
\begin{equation*}
0\longrightarrow \mathcal{F}\longrightarrow \mathbf{A}_{\mathcal{F}%
}^{0}\longrightarrow \mathbf{A}_{\mathcal{F}}^{1}\longrightarrow \cdots
\end{equation*}%
is a flasque resolution.
\end{theorem}

See Huber \cite{MR1105583}, \cite{MR1138291}, for a detailed proof. The
purpose of this paper is to find an analogous theorem for the sheaves which
come from Rost cycle modules. To get an idea of the problem, let us
experiment with the $K$-theory sheaf $\mathcal{K}_{n}$. Then already for a
single local field as in eq. \ref{lX_7} there are two natural candidates for
the $\mathbf{Q}_{p}$ example:%
\begin{equation*}
K_{n}(\underrightarrow{\limfunc{colim}}_{j}\underleftarrow{\lim }_{i}p^{-j}%
\mathbf{Z}/p^{i}\mathbf{Z})\qquad \text{or}\qquad \underrightarrow{\limfunc{%
colim}}_{j}\underleftarrow{\lim }_{i}K_{n}(p^{-j}\mathbf{Z}/p^{i}\mathbf{Z})%
\text{.}
\end{equation*}%
We immediately notice that \textit{the right-hand side version does not even
make sense}. While the colimit over the system $p^{-j}\mathbf{Z}_{p}$ has a
ring structure and equals $\mathbf{Q}_{p}$, the individual terms are just $%
\mathbf{Z}_{p}$-modules, but not rings. Although this poses no problem for
Beilinson's ad\`{e}les of quasi-coherent sheaves, it deprives us from
interpreting the terms as schemes, and thus to speak of the $K$-theory of
them. Thus, only the left-hand side version remains as a feasible
definition. However, its global analogue is $K_{n}(\mathbf{A}_{\mathcal{O}%
_{X}}^{i}(X))$ and $\limfunc{Spec}\mathbf{A}_{\mathcal{O}_{X}}^{i}(X)$ is a
terrible space and bears no resemblance to classical id\`{e}le-type
constructions. For example, the classical Brauer-Hasse-Noether sequence%
\begin{equation*}
0\longrightarrow \limfunc{Br}K\longrightarrow \tcoprod\nolimits_{v}\limfunc{%
Br}\widehat{K}_{v}\longrightarrow \mathbf{Q}/\mathbf{Z}\longrightarrow 0
\end{equation*}%
for a number field $K$ is suggestive of the fact that invariants (like the
Brauer group) should be taken for local fields individually, rather than of
the whole ad\`{e}les; \textquotedblleft $\limfunc{Br}(\tcoprod\nolimits_{v}%
\widehat{K}_{v})$\textquotedblright\ is not reasonable.

\begin{remark}
M. Morrow has recently demonstrated that the groups $\underleftarrow{\lim }%
_{i}K_{n}(\mathbf{Z}/p^{i}\mathbf{Z})$ play an interesting role in a
corrected Gersten conjecture for singular schemes. Although not in this
paper, this suggests that they might play a role in future id\`{e}le
theories \cite{MMorrow2}, \cite{MMorrow1}.
\end{remark}

In order to resolve the problem, we need to understand how (higher) local
fields occur in the Beilinson ad\`{e}les:

\begin{definition}
\label{lB_MarkerDefHigherLocalField}(Parshin \cite{MR514485}) For $n\geq 1$
an $n$-\emph{local field} $F$ \emph{with last residue field} $k$ is a
complete discrete valuation field such that its residue field is an $(n-1)$%
-local field with last residue field $k$. We agree on the convention that $k$
is the only $0$-local field with last residue field $k$.
\end{definition}

The datum of a $2$-local field can be depicted graphically as 
\begin{equation}
\begin{array}{cccccl}
F_{2} &  &  &  &  & \text{(level }2\text{ field)} \\ 
\uparrow &  &  &  &  &  \\ 
A_{2} & \longrightarrow & F_{1} &  &  & \text{(level }1\text{ field)} \\ 
&  & \uparrow &  &  &  \\ 
&  & A_{1} & \longrightarrow & F_{0}\text{,} & \text{(last residue field)}%
\end{array}
\label{lBT_5}
\end{equation}%
with the arrows $A_{i}\rightarrow F_{i}$ being the inclusions of the rings
of integers $A_{i}$ into their fields of fractions, $A_{i}\rightarrow
F_{i-1} $ the quotient maps.

\begin{theorem}
\label{TATE_StructureOfLocalAdelesProp}(Beilinson \cite{MR565095}) Suppose $%
X $ is a $n$-dimensional reduced scheme of finite type over a field $k$ and $%
\triangle =\{(\eta _{0}\geq \cdots \geq \eta _{n})\}$ a singleton set such
that $\limfunc{codim}_{X}\eta _{i}=i$. Then $A(\triangle ,\mathcal{O}_{X})$
is a finite direct product of $n$-local fields $\prod K_{j}$.
\end{theorem}

See \cite[p. $2$, second paragraph]{MR565095} or again Huber \cite{MR1105583}%
, \cite{MR1138291}. Dropping the assumption $\limfunc{codim}_{X}\eta _{i}=i$
one finds similar structures, e.g. fields of fractions of complete local
rings whose residue fields are again of this type $-$ or eventually a finite
extension of $k$. We will not go into details.

\begin{lemma}
If $F$ is an $n$-local field with last residue field $k$ and $\limfunc{char}%
F=\limfunc{char}k$, then there is a non-canonical field isomorphism%
\begin{equation*}
F\simeq k^{\prime }((t_{1}))\cdots ((t_{n}))
\end{equation*}%
for $k^{\prime }/k$ a finite field extension.
\end{lemma}

This follows from a repeated use of Cohen's Structure Theorem in the
equicharacteristic case. We will not use this lemma, but it might be
instructive to see that the $A(\triangle ,\mathcal{O}_{X})$ for singletons $%
\triangle $ have a very uniform structure. See \cite{MR1804915} for more
background.\medskip \newline
\textit{Now we focus on the case of }$X/k$\textit{\ an integral smooth
surface}. In order to proceed we need to make $A(\triangle ,\mathcal{O}_{X})$
explicit for $\triangle $ any singleton set.

\begin{lemma}
\label{computelocalcomps}Write $\eta \in X^{(0)}$ for the generic point and $%
y\in X^{(1)}$ and $x\in X^{(2)}$ with $x\in \overline{\{y\}}$ are arbitrary.
One computes the following:

\begin{enumerate}
\item $A(\{x\},\mathcal{O}_{X})=\widehat{\mathcal{O}}_{x}$ is the completion
of the local ring $\mathcal{O}_{x}$.\newline
This is a $2$-dimensional complete regular local domain.

\item $A(\{y\},\mathcal{O}_{X})=\widehat{\mathcal{O}}_{y}$ is the completion
of the local ring $\mathcal{O}_{y}$.\newline
This is a $1$-dimensional complete regular local domain. Hence, it is a
complete DVR with residue field $\kappa \left( y\right) $.

\item $A(\{y\geq x\},\mathcal{O}_{X})=\prod \widehat{\mathcal{O}}%
_{x,y_{i}^{\prime }}$ , where

\begin{itemize}
\item $y_{i}^{\prime }$ runs through the finitely many\footnote{%
They correspond to points over $x$ in the normalization of $\overline{\{y\}}$%
, see \cite[\S 6, Thm. 6.5 (4)]{MR0241408}. In particular there are only
finitely many such.} primes in $\widehat{\mathcal{O}}_{x}$ over $y$, and

\item $\widehat{\mathcal{O}}_{x,y_{i}^{\prime }}$ denotes the $y_{i}^{\prime
}$-adic completion of the localization of $\widehat{\mathcal{O}}_{x}$ at $%
y_{i}^{\prime }$, one could also write $\widehat{(\widehat{\mathcal{O}}%
_{x})_{y_{i}^{\prime }}}$. This ring is a $1$-dimensional complete regular
local domain. Thus, it is a complete DVR and the residue field is $\kappa
(y_{i}^{\prime })$ of $\limfunc{Spec}\widehat{\mathcal{O}}_{x}$.
\end{itemize}

\item $A(\{\eta \},\mathcal{O}_{X})=K:=\kappa (\eta )$ is the function field
of $X$.

\item $A(\{\eta \geq y\},\mathcal{O}_{X})=\widehat{K}_{y}:=\limfunc{Frac}%
\widehat{\mathcal{O}}_{y}$ is the field of fractions of $\widehat{\mathcal{O}%
}_{y}$.\newline
This is a $1$-local field with last residue field $\kappa \left( y\right) $.

\item $A(\{\eta \geq x\},\mathcal{O}_{X})=\widehat{\mathcal{O}}_{x}\otimes K$%
.

\item $A(\{\eta \geq y\geq x\},\mathcal{O}_{X})=\prod \widehat{K}%
_{x,y_{i}^{\prime }}$, where $\limfunc{Frac}\widehat{\mathcal{O}}%
_{x,y_{i}^{\prime }}$ is the field of fractions of $\widehat{\mathcal{O}}%
_{x,y_{i}^{\prime }}$. Each $\widehat{K}_{x,y_{i}^{\prime }}$ is a $2$-local
field with last residue field $\kappa (x)$. Mimicking Fig. \ref{lBT_5} its
structure is given by%
\begin{equation}
\begin{array}{ccccc}
\widehat{K}_{x,y_{i}^{\prime }} &  &  &  &  \\ 
\uparrow &  &  &  &  \\ 
\widehat{\mathcal{O}}_{x,y_{i}^{\prime }} & \longrightarrow & \kappa
(y_{i}^{\prime }) &  &  \\ 
&  & \uparrow &  &  \\ 
&  & (\widehat{\mathcal{O}}_{x}/y_{i}^{\prime }) & \longrightarrow & \kappa
\left( x\right) \text{.}%
\end{array}
\label{l346}
\end{equation}%
(Writing $\widehat{\mathcal{O}}_{x}/y_{i}^{\prime }$ we identify the point $%
y_{i}^{\prime }$ with its underlying prime ideal.)

\item $\widehat{K}_{x}:=\limfunc{Frac}\widehat{\mathcal{O}}_{x}$ (this does 
\textit{not} appear in the ad\`{e}les, but we fix the notation!)
\end{enumerate}
\end{lemma}

This computation is elementary with possibly one exception: The fact that
there are products appearing in $\prod \widehat{\mathcal{O}}%
_{x,y_{i}^{\prime }}$ or $\prod \widehat{K}_{x,y_{i}^{\prime }}$ relies on
the behaviour of completion at singularities. The products reduce to a 
\textit{single} factor whenever $\overline{\{y\}}$ is a \textit{normal}
scheme. In general there are as many factors as there are preimages of $x$
in the fiber of the normalization $\overline{\{y\}}^{\prime }\rightarrow 
\overline{\{y\}}\ni x$. This is explained in \cite[\S 6, Thm. 6.5 (4)]%
{MR0241408}, \cite[\S 1.1]{MR697316}.

\begin{example}
Consider $X=\mathbf{A}_{k}^{2}=\limfunc{Spec}k[s,t]$ and $\triangle =\{\eta
\geq y\geq x\}$ with $\eta =(0)$, $y=(s^{3}+s^{2}-t^{2})$, $x=(s,t)$ in
terms of prime ideals. The curve $\overline{\{y\}}$ is singular at $x$ and
the point has two preimages in the normalization $\overline{\{y\}}^{\prime }$%
. Correspondingly, $A(\triangle ,\mathcal{O}_{X})=\widehat{K}%
_{x,y_{1}^{\prime }}\oplus \widehat{K}_{x,y_{2}^{\prime }}$.
\end{example}

\begin{definition}
A height one prime in $\widehat{\mathcal{O}}_{x}$ is called a \emph{%
transcendental curve} if it does not lie in the fiber of a height one prime
under $\limfunc{Spec}\widehat{\mathcal{O}}_{x}\rightarrow \limfunc{Spec}%
\mathcal{O}_{x}$.
\end{definition}

\begin{example}
For $k=\mathbf{Q}$, $\mathbf{A}_{\mathbf{Q}}^{2}=\limfunc{Spec}\mathbf{Q}%
[s,t]$, $x=(0,0)$ and all $i\geq 1$ the primes $(s-t^{i}\exp t)$ constitute
an infinite set of transcendental curves in $\widehat{\mathcal{O}}_{x}\simeq 
\mathbf{Q}[[s,t]]$.
\end{example}

\begin{lemma}
\label{idelelemma_genericfibercomplete}$A(\{\eta \geq x\},\mathcal{O}_{X})=%
\widehat{\mathcal{O}}_{x}\otimes K$ is a principal ideal domain with
infinitely many maximal ideals. They correspond bijectively to the
transcendental curves in $\widehat{\mathcal{O}}_{x}$.
\end{lemma}

\begin{proof}
$\widehat{\mathcal{O}}_{x}$ is two-dimensional. It is regular, thus
factorial, and so all height one primes are principal. Tensoring with $K$
kills the maximal ideal (among others), so the remaining height one primes
become maximal. Then use the Example (replace $\exp t$ by another
transcendental function if $\limfunc{char}k>0$). Note that $\widehat{%
\mathcal{O}}_{x}\otimes K$ is the fiber of the generic point under $\limfunc{%
Spec}\widehat{\mathcal{O}}_{x}\rightarrow \limfunc{Spec}\mathcal{O}_{x}$,
giving the second claim.
\end{proof}

We shall need the following general \textquotedblleft Gersten
conjecture\textquotedblright -type property for big cycle modules:

\begin{proposition}
\label{Prop_GerstenConjForRostComplex}(Rost \cite[Thm. 6.1]{MR1418952})
Suppose $X:=\limfunc{Spec}\mathcal{O}_{Y,x}$ is the spectrum of a local ring
of a smooth scheme $Y/k$ and let $M_{\ast }$ be a (big) cycle module. Then
the cycle complex $C^{\bullet }\left( X,M_{\ast }\right) $ as in eq. \ref%
{l404} is exact.
\end{proposition}

This shows that the sheafification of eq. \ref{l404} provides a flasque
resolution of $\mathcal{M}_{\ast }$, \cite[Cor. 6.5]{MR1418952}. Moreover,
we need the analogous statement for complete rings. Such a statement is not
proven in \cite{MR1418952} (and would not fit in with the axioms in \textit{%
loc. cit.}), but the general principle of proof is well-known since
Quillen's breakthrough paper on algebraic $K$-theory \cite{MR0338129} and
generalizes to cycle modules:

\begin{proposition}
\label{Prop_GerstenConjForRostComplexCompleteVersion}Suppose $X/k$ is an
integral equicharacteristic complete regular local scheme and let $M_{\ast }$
be a big cycle module. Then the cycle complex $C^{\bullet }\left( X,M_{\ast
}\right) $ as in eq. \ref{l404} is exact.
\end{proposition}

We defer the proof to \S \ref{section_AppendixCompleteGersten}. As a result:

\begin{proposition}
\label{prop_alllocalidelesgersten}The complex $C^{\bullet }\left( \limfunc{%
Spec}A(\triangle ,\mathcal{O}_{X}),M_{\ast }\right) $ is exact for all rings 
$A(\triangle ,\mathcal{O}_{X})$ appearing in the list in Lemma \ref%
{computelocalcomps}.
\end{proposition}

\begin{proof}
For the $A(\triangle ,\mathcal{O}_{X})$ which are fields this is obvious;
for the complete local rings use Prop. \ref%
{Prop_GerstenConjForRostComplexCompleteVersion}. Only $A(\{\eta \geq x\},%
\mathcal{O}_{X})=\widehat{\mathcal{O}}_{x}\otimes K$ requires an argument.
The diagram%
\begin{equation}
\begin{xy} \xymatrix{ M_{\ast }(\widehat{K}_{x}) \ar[r] &
\coprod_{\widehat{y}}M_{\ast -1}(\kappa (\widehat{y})) \ar[r] & 0 & \\
M_{\ast }(\widehat{K}_{x}) \ar[r] \ar[u]_{\cong} & \coprod_{y^{\prime
}}M_{\ast -1}(\kappa (y^{\prime})) \ar[r] \ar[u] & M_{\ast -2}(\kappa (x))
\ar[r] & 0 \\ M_{\ast }(K) \ar[r] \ar[u] & \coprod_{y}M_{\ast -1}(\kappa
(y)) \ar[r] \ar[u] & M_{\ast -2}(\kappa (x)) \ar[r] \ar[u]_{\cong} & 0. }
\end{xy}  \label{l441}
\end{equation}%
commutes, where the the top row is the cycle complex in question, so by
Lemma \ref{idelelemma_genericfibercomplete} the variable $\hat{y}$ runs
through the transcendental curves of $\widehat{\mathcal{O}}_{x}$; the middle
is the cycle complex of $\widehat{\mathcal{O}}_{x}$ (so that $y^{\prime }$
runs through all height one primes) and exact by Prop. \ref%
{Prop_GerstenConjForRostComplexCompleteVersion}; the bottom row the one of $%
\mathcal{O}_{x}$ and exact by Prop. \ref{Prop_GerstenConjForRostComplex}.
The vertical arrows are functorially induced from the flat pullbacks along $%
\limfunc{Spec}\widehat{\mathcal{O}}_{x}\otimes K\rightarrow \limfunc{Spec}%
\widehat{\mathcal{O}}_{x}\rightarrow \limfunc{Spec}\mathcal{O}_{x}$. Only
surjectivity in the top row needs an argument. Map the elements $(\alpha _{%
\hat{y}})\in M_{\ast -1}(\kappa (\widehat{y}))$ to the middle row $-$ the
upward arrow is a projector on a subset of summands; so it is canonical
split by the corresponding inclusion. If they map to a non-zero element in $%
M_{\ast -2}(\kappa (x))$, use surjectivity in the last row to find a lift to 
$\coprod_{y}M_{\ast -1}(\kappa (y))$. Change the element in the middle row
by this lift so that it maps to zero in $M_{\ast -2}(\kappa (x))$. Exactness
of the middle row yields a lift in $M_{\ast }(\widehat{K}_{x})$, proving
surjectivity in the top row by commutativity of the diagram.
\end{proof}

\section{\label{section_IdelesNoProduct}Id\`{e}les with cycle module
coefficients}

As in the previous section $X/k$ is an integral smooth surface and $M_{\ast
} $ a big cycle module as in \S \ref{section_CycleModules}.\medskip \newline
Contrary to \S \ref{section_IdelesWithQCCoeffs} we will speak of \textit{id%
\`{e}les} instead of \textit{ad\`{e}les }from now on. The conventions for
this in the literature are blurry. For a curve over a field the classical ad%
\`{e}les arise from the theory of \S \ref{section_IdelesWithQCCoeffs}, which
is a good motivation to speak of `ad\`{e}les'. The classical id\`{e}les
could be seen as the degree one part of $K$-theory id\`{e}les in the sense
of the present theory, so we call them `id\`{e}les'.

As before, write $\eta \in X^{(0)}$ for the generic point and $y\in X^{(1)}$
and $x\in X^{(2)}$ with $x\in \overline{\{y\}}$ are arbitrary.

\begin{notation}
We shall write%
\begin{equation}
\prod\nolimits_{y}M_{\ast }(-)\text{\quad as a shorthand for\quad }%
F:U\rightarrow \prod\nolimits_{y\in U^{\left( 1\right) }}M_{\ast }(-)\text{,}
\label{l403}
\end{equation}%
where the latter is a (flasque) Zariski sheaf of $\mathbf{Z}$-graded abelian
groups. Analogously for $\prod_{x}$, where $x$ runs through all closed
points $U^{\left( 2\right) }$, e.g. $U\rightarrow \prod_{x\in U^{\left(
2\right) }}M_{\ast }(-)$. Writing $\prod_{x,y_{i}^{\prime }}$, the second
parameter $y_{i}^{\prime }$ runs through the height one primes in $\widehat{%
\mathcal{O}}_{x}$ except the transcendental curves, e.g. $U\rightarrow
\prod_{x\in U^{\left( 2\right) }}\prod_{y_{i}^{\prime }}M_{\ast }(-)$ with $%
y_{i}^{\prime }$ any non-transcendental curve in $\widehat{\mathcal{O}}_{x}$.
\end{notation}

\begin{example}
$\prod_{x,y_{i}^{\prime }}M_{\ast }(\widehat{\mathcal{O}}_{x,y_{i}^{\prime
}})$ denotes the flasque sheaf%
\begin{equation*}
U\rightarrow \tprod\nolimits_{x\in U^{\left( 2\right)
}}\tprod\nolimits_{y_{i}^{\prime }\in (\limfunc{Spec}\widehat{\mathcal{O}}%
_{x})^{(1)}\text{, non-transc}.}M_{\ast }(\widehat{\mathcal{O}}%
_{x,y_{i}^{\prime }})\text{.}
\end{equation*}
\end{example}

Recall that $M_{\ast }(-)$ for a ring is defined by eq. \ref%
{l322_DefOfUnramCycles}. These products are supposed to mimick the product
appearing in eq. \ref{TATEMATRIX_l6}. We have seen in Example \ref%
{Example_AutomaticFinitenessCondition} that the Beilinson ad\`{e}les
automatically impose finiteness conditions as they occur for example in eq. %
\ref{lX_2} and eq. \ref{lX_9}. As explained in \S \ref%
{section_IdelesWithQCCoeffs}, a literal translation of Beilinson's ad\`{e}%
les for quasi-coherent sheaves to $K$-theory sheaves does not work. The same
applies to more general cycle module coefficients. Thus, we will use the
cycle module associated to the local components $A(\triangle ,\mathcal{O}%
_{X})$ for $\triangle $ a singleton set, but patch them together manually to
imitate (vaguely!) the finiteness conditions as they would come from
Beilinson's construction. Naturally, carrying out this process `by hand'
requires a slightly unpleasant amount of individual definitions:

\begin{notation}
The symbol $\tilde{\forall}x$ will mean: for all $x$ with possibly finitely
many exceptions. We also say \textquotedblleft for almost all $x$%
\textquotedblright .
\end{notation}

\begin{definition}
\label{recicircleconditionDef}For the sheaves as introduced in eq. \ref{l403}%
, we introduce the following variations which impose manual finiteness
conditions:

\begin{enumerate}
\item The prime superscript in $\prod_{y}^{\prime }M_{\ast }(\widehat{K}%
_{y}) $ means that we consider only those elements $(\alpha _{y})_{y}$ such
that $\tilde{\forall}y\in X^{(1)}:$ $\alpha _{y}\in M_{\ast }(\widehat{%
\mathcal{O}}_{y})$.

\item The triple prime superscript in $\prod_{x,y_{i}^{\prime }}^{\prime
\prime \prime }M_{\ast }(\widehat{K}_{x,y_{i}^{\prime }})$ means that we
consider only those elements $(\alpha _{x,y_{i}^{\prime }})_{x,y_{i}^{\prime
}}$ such that

\begin{description}
\item[(a)] $\forall x\,\tilde{\forall}y_{i}^{\prime }:\alpha
_{x,y_{i}^{\prime }}\in M_{\ast }(\widehat{\mathcal{O}}_{x,y_{i}^{\prime }})$%
.

\item[(b)] \label{recicirclecondition}$\tilde{\forall}x:\sum_{y_{i}^{\prime
}}\partial _{x}^{y_{i}^{\prime }}\partial _{y_{i}^{\prime }}^{\widehat{K}%
_{x,y_{i}^{\prime }}}(\alpha _{x,y_{i}^{\prime }})=0\in M_{\ast -2}(\kappa
(x))$.
\end{description}
\end{enumerate}
\end{definition}

The condition \textbf{(a)} ensures that the sum in \textbf{(b)} is finite.

\begin{example}
Note that (a) still allows $\alpha _{x,y_{i}^{\prime }}\notin M_{\ast }(%
\widehat{\mathcal{O}}_{x,y_{i}^{\prime }})$ for infinitely many $%
y_{i}^{\prime }$ as shown in this example:%

\[
{\includegraphics[
height=0.5419in,
width=1.1102in
]%
{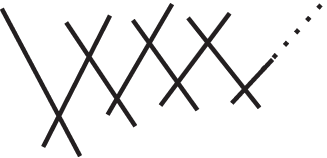}%
}
\]

The lines represent points $y_{i}^{\prime }$ with $\alpha _{x,y_{i}^{\prime
}}\notin M_{\ast }(\widehat{\mathcal{O}}_{x,y_{i}^{\prime }})$. In this
particular example infinitely many closed points $x$ are involved, so the
condition (b) is non-trivial.
\end{example}

\textit{Tricomplexes:} For every abelian category $P$ let $CP$ denote the
abelian category of bounded complexes in $P$. Then $CCP$ can be identified
with the category of bounded bicomplexes. There is the standard total
complex functor $CCP\rightarrow CP$. Applying this to $CP$, $CCCP$ can be
viewed as the abelian category of \emph{tricomplexes} and each tricomplex
has a natural total complex given by the composition $CC(CP)\rightarrow
CCP\rightarrow CP$. The sign correction of the differential in the total
complex then induces the correct sign for the given complex. Depending on
one's preferred sign convention, the overall sign in explicit formulas might
come out opposite.

\begin{definition}
\label{def_base_ideles}We define a complex of sheaves of $\mathbf{Z}$-graded
abelian groups%
\begin{equation}
0\longrightarrow \mathcal{M}_{\ast }\longrightarrow \mathbf{I}_{\mathcal{M}%
}^{0}\longrightarrow \mathbf{I}_{\mathcal{M}}^{1}\longrightarrow \mathbf{I}_{%
\mathcal{M}}^{2}\longrightarrow 0  \label{l314}
\end{equation}%
as the total complex of the following tricomplex:%
\begin{equation}
\begin{xy} \xymatrix@!0@C=24mm@R=9mm{ &
{\prod_{{x}}M_{\ast}(\widehat{\mathcal{O}}_{{x}}\otimes K)} \ar[dl] & &
{\prod_{{x}}M_{\ast }(\widehat{\mathcal{O}}_{{x}}),} \ar[dl] \ar[ll] \\
{\prod_{{x},{y}_{i}^{\prime}}^{\prime\prime\prime}M_{\ast}(%
\widehat{K}_{{x},{y}_{i}^{\prime}})} & &
{\prod_{{x},{y}_{i}^{\prime}}M_{\ast
}(\widehat{\mathcal{O}}_{{x},{y}_{i}^{\prime}})} \ar [ll] & \\ &
{\coprod\nolimits_{\eta}M_{\ast}(K)} \ar[dl] \ar '[u][uu] & &
{\mathcal{M}_{\ast}} \ar[dl] \ar'[l][ll] \ar[uu] \\
{\prod_{{y}}^{\prime}M_{\ast}(\widehat{K}_{{y}})} \ar [uu] & &
{\prod_{{y}}M_{\ast }(\widehat{\mathcal{O}}_{{y}})} \ar[uu] \ar[ll] & }
\end{xy}  \label{lX_200}
\end{equation}%
The arrows are induced from the flat pullbacks along the respective
completion and localization maps.
\end{definition}

As $X$ has only a single generic point $\eta $, $\prod\nolimits_{\eta
}M_{\ast }(K)$ is a redundant formulation, but it stresses the general
pattern.

\begin{remark}
The rings $\widehat{\mathcal{O}}_{x}$, $\widehat{\mathcal{O}}%
_{x,y_{i}^{\prime }}$,\ldots\ appearing in the tricomplex are precisely the
local components $A(-,\mathcal{O}_{X})$ of the Beilinson ad\`{e}les with
coefficients in the structure sheaf $\mathcal{O}_{X}$, see Lemma \ref%
{computelocalcomps} (albeit with the degenerate simplices removed).
Moreover, the complex $\mathbf{A}_{\mathcal{O}_{X}}^{\bullet }$ of Thm. \ref%
{lX_BeilinsonResolutionThm} can $-$ (not quite, but) roughly $-$ be written
down as the total complex of a tricomplex resembling the one in Fig. \ref%
{lX_200}.
\end{remark}

\begin{lemma}
The morphisms in the tricomplex in eq. \ref{lX_200} are well-defined.
\end{lemma}

\begin{proof}
The claim is obvious, except possibly for checking that the individual
finiteness conditions are being respected. But this is straightforward.
\end{proof}

This construction boils down to a complex%
\begin{eqnarray}
\mathcal{M}_{\ast } &\rightarrow &\prod\nolimits_{\eta }M_{\ast }(K)\oplus
\prod\nolimits_{y}M_{\ast }(\widehat{\mathcal{O}}_{y})\oplus
\prod\nolimits_{x}M_{\ast }(\widehat{\mathcal{O}}_{x})  \notag \\
&\rightarrow &\prod\nolimits_{y}^{\prime }M_{\ast }(\widehat{K}_{y})\oplus
\prod\nolimits_{x}M_{\ast }(\widehat{\mathcal{O}}_{x}\otimes K)\oplus
\prod\nolimits_{x,y_{i}^{\prime }}M_{\ast }(\widehat{\mathcal{O}}%
_{x,y_{i}^{\prime }})  \label{l9622} \\
&\rightarrow &\prod\nolimits_{x,y_{i}^{\prime }}^{\prime \prime \prime
}M_{\ast }(\widehat{K}_{x,y_{i}^{\prime }})\rightarrow 0\text{.}  \notag
\end{eqnarray}%
Except $\mathcal{M}_{\ast }$ the sheaves are all flasque. If $M_{\ast }$ is
understood, we write $\mathbf{I}^{i}$ instead of $\mathbf{I}_{\mathcal{M}%
}^{i}$. We shall also write $\mathbf{I}^{0}=\mathbf{I}^{(0)}\oplus \mathbf{I}%
^{(1)}\oplus \mathbf{I}^{(2)}$; $\mathbf{I}^{1}=\mathbf{I}^{(01)}\oplus 
\mathbf{I}^{(02)}\oplus \mathbf{I}^{(12)}$; $\mathbf{I}^{2}=\mathbf{I}%
^{(012)}$ according to the direct sum decompositions of the three terms in
eq. \ref{l9622} (here the superscripts indicate the codimensions of the
points appearing in the respective flags). This is convenient to stress what
component sections of these sheaves are associated to; we will usually use
superscripts using this indexing as in $(f^{0},f_{y}^{1},f_{x}^{2})$; $%
(f_{y}^{01},f_{x}^{02},f_{x,y_{i}^{\prime }}^{12})$; $(f_{x,y_{i}^{\prime
}}^{012})$ respectively. The maps all turn out to be induced from signed
diagonal maps and the appropriate restrictions, e.g. the first arrow locally
unwinds as%
\begin{equation*}
f\mapsto (f,\prod_{y}\limfunc{res}\nolimits_{K}^{\widehat{K}_{y}}f,\prod_{x}%
\limfunc{res}\nolimits_{K}^{\widehat{K}_{x}}f)\text{,}
\end{equation*}%
the last as $(f_{y}^{01},f_{x}^{02},f_{x,y_{i}^{\prime }}^{12})\mapsto
(f_{x,y_{i}^{\prime }}^{12}-\limfunc{res}\nolimits_{\widehat{K}_{x}}^{%
\widehat{K}_{x,y_{i}^{\prime }}}f_{x}^{02}+\limfunc{res}\nolimits_{\widehat{K%
}_{y}}^{\widehat{K}_{x,y_{i}^{\prime }}}f_{y}^{01})$.

Write $A^{i}(-,M_{\ast }):=H^{i}(C^{\bullet }(-,M_{\ast }))$ to denote the
cohomology of the Rost cycle complex.

\begin{theorem}
(Main Theorem) \label{marker_MAINTHM}Let $X/k$ be a smooth integral surface, 
$M_{\ast }$ a big cycle module and $\mathcal{M}_{\ast }$ the associated
Zariski sheaf.

\begin{enumerate}
\item \label{mainthmp1}Then%
\begin{equation*}
\mathcal{M}_{\ast }\longrightarrow \mathbf{I}_{\mathcal{M}%
}^{0}\longrightarrow \mathbf{I}_{\mathcal{M}}^{1}\longrightarrow \mathbf{I}_{%
\mathcal{M}}^{2}\longrightarrow 0
\end{equation*}%
is a flasque resolution.

\item \label{mainthmp2}The isomorphisms%
\begin{equation*}
\alpha ^{i}:H^{i}(X,\mathbf{I}_{\mathcal{M}}^{\bullet })\rightarrow
A^{i}(X,M_{\ast })
\end{equation*}%
are explicitly given by%
\begin{eqnarray}
\alpha ^{0} &:&(f^{0},f_{y}^{1},f_{x}^{2})\mapsto f^{0}\in M_{\ast }(\kappa
(\eta ))  \notag \\
\alpha ^{1} &:&(f_{y}^{01},f_{x}^{02},f_{x,y_{i}^{\prime }}^{12})\mapsto
\partial _{y}^{\widehat{K}_{y}}(f_{y}^{01})\in M_{\ast -1}(\kappa (y))
\label{lX_10} \\
\alpha ^{2} &:&(f_{x,y_{i}^{\prime }}^{012})\mapsto \sum_{y_{i}^{\prime
}}(\partial _{x}^{y_{i}^{\prime }}\circ \partial _{y_{i}^{\prime }}^{%
\widehat{K}_{x,y_{i}^{\prime }}})(f_{x,y_{i}^{\prime }}^{012})\in M_{\ast
-2}(\kappa (x))  \label{lX_11}
\end{eqnarray}
\end{enumerate}
\end{theorem}

Part \ref{mainthmp1} is the analogue of Beilinson's resolution of
quasi-coherent sheaves in Thm. \ref{lX_BeilinsonResolutionThm}. The rest of
the section is devoted to the proof. We fix $M_{\ast }$ and write $\mathbf{I}%
^{i}$ instead of $\mathbf{I}_{\mathcal{M}}^{i}$.\newline
Firstly, we define 'id\`{e}les with a reciprocity constraint'. These sheaves
are defined as kernels:%
\begin{eqnarray*}
0 &\longrightarrow &\prod\nolimits_{y}^{\prime +\text{recip.}}M_{\ast }(%
\widehat{K}_{y})\longrightarrow \prod\nolimits_{y}^{\prime }M_{\ast }(%
\widehat{K}_{y})\overset{\partial _{x}^{y}\partial _{y}}{\longrightarrow }%
\coprod\nolimits_{x}M_{\ast -2}(\kappa (x)) \\
0 &\longrightarrow &\prod\nolimits_{x,y_{i}^{\prime }}^{\prime \prime \prime
+\text{recip.}}M_{\ast }(\widehat{K}_{x,y_{i}^{\prime }})\longrightarrow
\prod\nolimits_{x,y_{i}^{\prime }}^{\prime \prime \prime }M_{\ast }(\widehat{%
K}_{x,y_{i}^{\prime }})\overset{\partial _{x}^{y_{i}^{\prime }}\partial
_{y_{i}^{\prime }}}{\longrightarrow }\coprod\nolimits_{x}M_{\ast -2}(\kappa
(x))
\end{eqnarray*}%
We shall keep this notation for later use below. Here\ `reciprocity
constraint' refers to the fact that an element in $M_{\ast }(K)$, when being
mapped diagonally to $\prod\nolimits_{y}^{\prime }M_{\ast }(\widehat{K}_{y})$%
, automatically satisfies $\sum_{y}\partial _{x}^{y}\partial _{y}x=0$. As
this sum runs over all curves $y$ through a fixed closed point $x$, this is
sometimes called a `reciprocity law around a point'.

Now we focus on the bicomplex which constitutes the bottom face of the
tricomplex of Def. \ref{def_base_ideles}. The total complex of this
bicomplex turns out to be homologically concentrated in a single degree:

\begin{lemma}
\label{Lemma_FrontAdeles2DAreExact}The sequence of sheaves%
\begin{equation*}
0\longrightarrow \mathcal{M}_{\ast }\longrightarrow K\oplus
\prod\nolimits_{y}M_{\ast }(\widehat{\mathcal{O}}_{y})\longrightarrow
\prod\nolimits_{y}^{\prime +\text{recip.}}M_{\ast }(\widehat{K}%
_{y})\longrightarrow 0
\end{equation*}%
is exact.
\end{lemma}

We may rephrase this as follows: the bottom face of our tricomplex is quasi-isomorphic to the
sheaf $\coprod\nolimits_{x}M_{\ast -2}(\kappa (x))$, placed in degree two.

\begin{proof}
The injectivity is clear since for each open set $U$ the group of sections $%
\mathcal{M}_{\ast }(U)$ is defined as a subgroup of $M_{\ast }(K)$.
Exactness in the middle is easy: Let $x$ be a point in $X$, $U_{x}\ni x$
some open neighbourhood. Suppose local sections $\alpha \in M_{\ast }(K)$
and $\alpha _{y}\in M_{\ast }(\widehat{\mathcal{O}}_{y})$ over $U_{x}$ are
given and go to zero on the right. Then for all codimension one points $y\in
U_{x}^{\left( 1\right) }$ we have $\limfunc{res}\nolimits_{K}^{\widehat{K}%
_{y}}\alpha -\alpha _{y}=0\in M_{\ast }(\widehat{K}_{y})$. Thus, taking the
boundary at $y$ in $\widehat{\mathcal{O}}_{y}$ we obtain%
\begin{equation*}
0=\partial _{y}^{\widehat{K}_{y}}\limfunc{res}\nolimits_{K}^{\widehat{K}%
_{y}}\alpha -\partial _{y}^{\widehat{K}_{y}}\alpha _{y}=\limfunc{res}%
\nolimits_{\kappa \left( y\right) }^{\kappa \left( y\right) }\partial
_{y}^{K}\alpha -\partial _{y}^{\widehat{K}_{y}}\alpha _{y}
\end{equation*}%
by the compatibility of flat pullbacks (here already made explicit as $%
\limfunc{res}\nolimits_{K}^{\widehat{K}_{y}}$) with the differential in the
cycle complex, axiom \textbf{R3a}. We also have $\partial _{y}^{\widehat{K}%
_{y}}\alpha _{y}=0$ since $\alpha _{y}\in M_{\ast }(\widehat{\mathcal{O}}%
_{y})$ and so we deduce $\partial _{y}^{K}\alpha =0$. However, this shows
that $\alpha \in \mathcal{M}_{\ast }\left( U_{x}\right) $. For the
surjectivity suppose we are given $(\alpha _{y})_{y}$ on the right-hand
side. The cycle complex of $\limfunc{Spec}\mathcal{O}_{x}$ is exact by Prop. %
\ref{Prop_GerstenConjForRostComplex}:%
\begin{equation*}
0\rightarrow M_{\ast }(\mathcal{O}_{x})\rightarrow M_{\ast }(K_{\eta
})\rightarrow \coprod\nolimits_{y\in (\mathcal{O}_{x})^{\left( 1\right)
}}M_{\ast -1}(\kappa (y))\rightarrow M_{\ast -2}(\kappa (x))\rightarrow 0%
\text{,}
\end{equation*}%
The element $(\partial _{y}^{\widehat{K}_{y}}\alpha _{y})$, placed in the
third term of this sequence, admits a preimage $f\in M_{\ast }(K)$ (\textit{%
Proof:} by the finiteness condition $\prod\nolimits^{\prime }$ this is
non-zero only for finitely many points $y$ and the reciprocity constraint
shows that the element goes to zero on the right). Secondly, note that $%
\beta _{y}:=\limfunc{res}\nolimits_{K}^{\widehat{K}_{y}}f-\alpha _{y}$
satisfies $\partial _{y}^{\widehat{K}_{y}}\beta _{y}=0$ for all $y$ in the
neighbourhood, so $\beta _{y}\in M_{\ast }(\widehat{\mathcal{O}}_{y})$. It
is now clear that $(f,\beta _{y})$ is a preimage of $\alpha _{y}$.
\end{proof}

Now we repeat this analysis with the top face of the tricomplex. Again, the
total complex of this bicomplex turns out to be homologically concentrated
in degree two:

\begin{lemma}
\label{Lemma_BackAdeles2DAreExact}The sequence of sheaves%
\begin{align*}
& 0\rightarrow \prod\nolimits_{x}M_{\ast }(\widehat{\mathcal{O}}%
_{x})\rightarrow \prod\nolimits_{x}M_{\ast }(\widehat{\mathcal{O}}%
_{x}\otimes K)\oplus \prod\nolimits_{x,y_{i}^{\prime }}M_{\ast }(\widehat{%
\mathcal{O}}_{x,y_{i}^{\prime }})\cdots \\
& \qquad \qquad \qquad \qquad \cdots \rightarrow
\prod\nolimits_{x,y_{i}^{\prime }}^{\prime \prime \prime +\text{recip.}%
}M_{\ast }(\widehat{K}_{x,y_{i}^{\prime }})\rightarrow 0
\end{align*}%
is exact.
\end{lemma}

In other words, the top face of our tricomplex is quasi-isomorphic to the
sheaf $\coprod\nolimits_{x}M_{\ast -2}(\kappa (x))$, placed in degree two.

\begin{proof}
The proof is very similar to the one of the previous lemma. Injectivity is
clear since $\widehat{\mathcal{O}}_{x}$ and $\widehat{\mathcal{O}}%
_{x}\otimes K$ both have field of fractions $\widehat{K}_{x}$. For exactness
in the middle, suppose we are given $\alpha _{x}\in M_{\ast }(\widehat{%
\mathcal{O}}_{x}\otimes K)$ and $\alpha _{x,y_{i}^{\prime }}\in M_{\ast }(%
\widehat{\mathcal{O}}_{x,y_{i}^{\prime }})$ going to zero on the right. This
unwinds as the cocycle condition%
\begin{equation}
\forall x,y_{i}^{\prime }:\limfunc{res}\nolimits_{\widehat{K}_{x}}^{\widehat{%
K}_{x,y_{i}^{\prime }}}\alpha _{x}-\alpha _{x,y_{i}^{\prime }}=0\in M_{\ast
}(\widehat{K}_{x,y_{i}^{\prime }})  \label{l372}
\end{equation}%
and applying the boundary $\partial _{y_{i}^{\prime }}^{\widehat{K}%
_{x,y_{i}^{\prime }}}$ yields $\limfunc{res}\nolimits_{\kappa (y^{\prime
})}^{\kappa (y^{\prime })}\partial _{y_{i}^{\prime }}^{\widehat{K}%
_{x}}\alpha _{x}=\partial _{y_{i}^{\prime }}^{\widehat{K}_{x,y_{i}^{\prime
}}}\alpha _{x,y_{i}^{\prime }}$. We conclude $\partial _{y_{i}^{\prime }}^{%
\widehat{K}_{x}}\alpha _{x}=0$. The height one primes of $\widehat{\mathcal{O%
}}_{x}$ disjointly decompose into transcendental $\widehat{y}$ and
non-transcendental $y_{i}^{\prime }$ curves. Since $\alpha _{x}\in M_{\ast }(%
\widehat{\mathcal{O}}_{x}\otimes K)$ we therefore know that the boundaries
at \textit{all} height one primes of $\widehat{\mathcal{O}}_{x}$ vanish and
thus $\alpha _{x}\in M_{\ast }(\widehat{\mathcal{O}}_{x})$. Finally, eq. \ref%
{l372} shows that $\alpha _{x}$ maps to $\alpha _{x,y_{i}^{\prime }}$ under
the first arrow, showing that we have found a preimage. To see surjectivity,
suppose we are given $\alpha _{x,y_{i}^{\prime }}\in M_{\ast }(\widehat{K}%
_{x,y_{i}^{\prime }})$. By the reciprocity constraint and the exactness of
the cycle complex of $\limfunc{Spec}\widehat{\mathcal{O}}_{x}$,%
\begin{equation*}
\begin{xy} \xymatrix{ \widehat{C}_{\bullet }: & M_{\ast
}(\widehat{\mathcal{O}}_{x}) \ar[r] & M_{\ast }(\widehat{K}_{x}) \ar[r] &
{\coprod\nolimits_{\tilde{y}}M_{\ast -1}(\kappa (\tilde{y}))} {
\save[]+<0cm,0.5cm>*{{\partial _{y_{i}^{\prime }}{\alpha}_{x,y_{i}^{\prime
}}} \in} \restore} \ar[r] & M_{\ast -2}(\kappa (x)) \ar[r] & 0 } \end{xy}
\end{equation*}%
we obtain a preimage $f\in M_{\ast }(\widehat{K}_{x})$. Since by
construction $\partial _{\widehat{y}}f=0$ at transcendental curves $\widehat{%
y}$ in $\widehat{\mathcal{O}}_{x}$, Prop. \ref{prop_alllocalidelesgersten}
tells us that $f\in M_{\ast }(\widehat{\mathcal{O}}_{x}\otimes K)$. It is
now obvious that $\beta _{x,y_{i}^{\prime }}:=\limfunc{res}\nolimits_{%
\widehat{K}_{x}}^{\widehat{K}_{x,y_{i}^{\prime }}}f-\alpha _{x,y_{i}^{\prime
}}\in M_{\ast }(\widehat{\mathcal{O}}_{x,y_{i}^{\prime }})$ and that $%
f\oplus \beta _{x,y_{i}^{\prime }}$ provides a preimage.
\end{proof}

Combining the two lemmata, we may quasi-isomorphically replace the top and
bottom face of our tricomplex so that we arrive at%
\begin{equation*}
\begin{xy} \xymatrix@!0@C=24mm@R=9mm{ & 0 \ar[dl] & & 0 \ar[dl] \ar[ll] \\
{\coprod_{x}M_{\ast -2}(\kappa (x))} & & 0 \ar [ll] & \\ & 0 \ar[dl] \ar
'[u][uu] & & 0 \ar[dl] \ar'[l][ll] \ar[uu] \\ {\coprod_{x}M_{\ast -2}(\kappa
(x))} \ar [uu]_{\cong} & & 0 \ar[uu] \ar[ll] & } \end{xy}
\end{equation*}%
A direct inspection reveals that the morphism along the front left edge is
indeed an isomorphism. However, this immediately implies that the total
complex of the tricomplex is acyclic. This finishes the proof of Thm. \ref%
{marker_MAINTHM}.\ref{mainthmp1}. The existence of some formula as in the
claim of Thm. \ref{marker_MAINTHM}.\ref{mainthmp2} is clear since the
flasque resolution property already implies \textit{abstractly} that%
\begin{equation*}
H^{i}(X,\mathbf{I}_{\mathcal{M}}^{\bullet })\overset{\cong }{\longrightarrow 
}A^{i}(X,M_{\ast })\text{.}
\end{equation*}%
One then finds the concrete formula by explicitly making the diagram chase
underlying this isomorphism (one just needs to go through the proofs of
Lemma \ref{Lemma_FrontAdeles2DAreExact} and Lemma \ref%
{Lemma_BackAdeles2DAreExact} again and keep track of the individual steps).

\begin{remark}
(Comparison with Gorchinskiy's theory) Statement and comparison maps in eq. %
\ref{lX_10} are entirely analogous to Gorchinskiy's comparison maps $\nu
_{\ast }$ \cite{MR2354210}, \cite[Thm. 1.1 and map in Prop. 2.16]{MR2489487}
for uncompleted ad\`{e}les. In \emph{loc. cit.} the Gersten complex is
called `Cousin complex' and the boundary maps are called `residue maps' as
in the analogue of the theory for quasi-coherent sheaves \cite{MR0222093}.
\end{remark}

\begin{remark}
\label{h1difffromintro}The sheaves $\mathbf{I}_{\mathcal{M}}^{\bullet }$ are
different from the $\mathbf{I},\mathbf{I}^{0}$ appearing in the introduction
in \S \ref{section_Motivation}. Using $K$-theory as a cycle module the above
yields a canonical isomorphism $\limfunc{CH}\nolimits^{1}(X)\cong H^{1}(X,%
\mathcal{K}_{1})\cong H^{1}(\mathbf{I}_{\mathcal{K},1}^{\bullet })$ (the
subscript `$1$' refers to taking degree one in the grading of the cycle
module). There is no harm using this isomorphism instead of the one in \S %
\ref{section_Motivation}. There is \emph{no way} the isomorphism in \S \ref%
{section_Motivation} generalizes to arbitrary cycle modules since eq. \ref%
{lX_3} is a resolution concentrated in degrees $[0,1]$ and if it generalized
this would imply $H^{2}(X,\mathcal{M}_{\ast })=0$ for all $\mathcal{M}_{\ast
}$, which is just false.
\end{remark}

As in \S \ref{section_IdelesNoProduct}, let $X/k$ be an integral smooth
surface over an arbitrary field. Let $M_{\ast }$ be a big cycle module with
a product, a pairing on itself as defined in \cite[Def. 2.1]{MR1418952}.
Such a pairing consists of a bilinear pairing of abelian groups%
\begin{equation}
\cdot :M_{p}\left( F\right) \times M_{q}\left( F\right) \longrightarrow
M_{p+q}\left( F\right) \text{.}  \label{l409}
\end{equation}%
Certain axioms \textbf{P1}-\textbf{P3} \cite{MR1418952} need to be
fulfilled. Such a pairing induces an associated pairing of cycle cohomology
groups.

\begin{example}
(Products) $M_{\ast }=K_{\ast }^{M}$ has such a pairing, just given by the
ordinary product in the Milnor $K$-theory ring. Similarly, Quillen $K$%
-theory and Galois cohomology with $\mathbf{Z}/\ell $ coefficients with its
cup product, $\ell $ coprime to $\limfunc{char}k$; $M_{i}(F):=H^{i}(F,%
\mathbf{Z}/\ell (i))$.
\end{example}

Given any such product, we obtain a graded-commutative product on the cycle
cohomology groups%
\begin{equation*}
A^{i}(X,M_{q})\otimes A^{j}(X,M_{r})\longrightarrow A^{i+j}(X,M_{q+r})\text{.%
}
\end{equation*}%
Using Milnor or Quillen $K$-theory as a cycle module, one has $%
A^{i}(X,K_{i})=\limfunc{CH}\nolimits^{i}(X)$ and the bidegree $(i,i)$
excerpt of the above product agrees with the usual commutative product of
the Chow ring. Feeding this into the main theorem, Thm. \ref{INTRO_MAINTHM},
we obtain a version of the commutative square in Fig. \ref{lX_5} in the
introduction \S \ref{section_Motivation}.%
\begin{equation*}
\begin{xy} \xymatrix{ \limfunc{CH}\nolimits^{1}(X) \otimes _{\mathbf{Z}}
\limfunc{CH}\nolimits^{1}(X) \ar[r] \ar[d] & \limfunc{CH}\nolimits^{2}(X)
\ar[d] \\ H^{1}(X,\mathbf{I}_{\mathcal{M}}^{\bullet })\otimes
_{\mathbf{Z}}H^{1}(X,\mathbf{I}_{\mathcal{M}}^{\bullet }) \ar[r]_-{\ast} &
H^{2}(X,\mathbf{I}_{\mathcal{M}}^{\bullet }) \\ } \end{xy}
\end{equation*}%
Note that it is actually slightly different from the one in \S \ref%
{section_Motivation}. The cohomology group $H^{1}(X,\mathbf{I}_{\mathcal{K}%
}^{\bullet })$ comes from a more complicated presentation than in \S \ref%
{section_Motivation}, but as already explained in Rmk. \ref{h1difffromintro}
this cannot be avoided if one wants a uniform id\`{e}le resolution for all
cycle module sheaves.

\begin{example}
(S.\ Gorchinskiy) I thank S. Gorchinskiy for communicating this example to
me. It would be desirable to lift the product $(\ast )$, which we have
essentially avoided to construct from scratch by transporting it from $%
A^{\ast }(X,M_{\ast })$, to a product defined on the id\`{e}le resolution
itself. The most straightforward definition would be%
\begin{equation}
\mathbf{I}^{p}\otimes _{\mathbf{Z}}\mathbf{I}^{q}\longrightarrow \mathbf{I}%
^{p+q}\text{;\quad }(e\otimes f)^{z_{0}\ldots z_{p+q}}:=\limfunc{res}%
\nolimits_{\ast }^{\ast }e^{z_{0}\ldots ,z_{p}}\cdot \limfunc{res}%
\nolimits_{\ast }^{\ast }f^{z_{p}\ldots z_{p+q}}  \label{l145}
\end{equation}%
where $\limfunc{res}\nolimits_{\ast }^{\ast }$ comes from the flat pullback
along the respective arrow in Fig. \ref{lX_200}. However, this definition
does \emph{not} work: Consider affine $2$-space $X:=\limfunc{Spec}k[s,t]$
and take Milnor $K$-theory as the cycle module. Then define%
\begin{equation*}
e_{y}^{01}:=\left\{ 
\begin{array}{cl}
s & \text{for }y=(s) \\ 
1 & \text{otherwise}%
\end{array}%
\right. \qquad f_{x,y}^{12}:=\left\{ 
\begin{array}{cl}
t-\alpha & \text{for }x=(s,t-\alpha )\text{, }\alpha \in k \\ 
& \text{and }y=(s)\text{.} \\ 
1 & \text{otherwise}%
\end{array}%
\right.
\end{equation*}%
in degree one (i.e. $K_{1}^{M}(F)=F^{\times }$). These id\`{e}les meet the
finiteness conditions of \S \ref{section_IdelesNoProduct}. However, $%
e\otimes f$ does not meet the finiteness condition of Def. \ref%
{recicircleconditionDef}.\ref{recicirclecondition}. And indeed an evaluation
of the comparison map in eq. \ref{lX_11} yields%
\begin{equation*}
\sum_{y}(\partial _{x}^{y}\circ \partial _{y}^{\widehat{K}_{x,y}})(e\otimes
f)_{x,y}=\partial _{(t-\alpha )}\partial _{(s)}\{s,t-a_{\alpha }\}=+1
\end{equation*}%
for \emph{all} points $x=(s,t-\alpha )$ ($\alpha $ arbitrary). This does not
lie in $\coprod_{x\in X^{(2)}}\mathbf{Z}$ (unless $k$ is finite). We may
even allow all $\alpha \in k[s]$. It is still reasonable to believe that
there exists a subcomplex of $\mathbf{I}_{\mathcal{M}}^{\bullet }$ on which
eq. \ref{l145} exhibits a well-defined product. Gorchinskiy constructed such
a subcomplex in his theory of adically non-completed id\`{e}les \cite%
{MR2354210}, \cite{MR2489487}.
\end{example}

\section{\label{section_AppendixCompleteGersten}Appendix: Gersten property
for complete rings}

Finally, we need to prove Prop. \ref%
{Prop_GerstenConjForRostComplexCompleteVersion}, which we repeat for
convenience:

\begin{proposition}
\label{Prop_GerstenConjForRostComplexCompleteVersion2}Suppose $X/k$ is an
integral equicharacteristic complete regular local scheme and let $M_{\ast }$
be a big cycle module as in \S \ref{section_CycleModules}. Then the complex $%
C^{\bullet }\left( X,M_{\ast }\right) $ is exact.
\end{proposition}

This crucial technical fact is independent of the rest of the paper.

We shall use the following lemma, variations of which occur in most proofs
of Noether normalization.

\begin{lemma}
\label{Lemma_NagataGoodCoordinatesForWDistinguished}(Nagata) Let $F$ be a
field. Suppose $f\in F[[w_{1},\ldots ,w_{n}]]$ is non-zero. Then there is a
ring automorphism $\sigma $ (fixing $F$) of the shape%
\begin{equation*}
w_{n}\mapsto w_{n}\text{\qquad }w_{i}\mapsto w_{i}+w_{n}^{c_{i}}
\end{equation*}%
for suitable $c_{i}\in \mathbf{Z}_{>0}$ and $i=1,\ldots ,n-1$ such that $%
\sigma f\in F[[w_{1},\ldots ,w_{n}]]$ is $w_{n}$-distinguished (i.e. the
coefficients as a power series in $R[[w_{n}]]$ with $R=F[[w_{1},\ldots
,w_{n-1}]]$ do not all lie in the ideal $(w_{1},\ldots ,w_{n-1})$ of $R$).
\end{lemma}

\begin{proof}
\cite[see proof of Prop. 1 in \S 5.2.4]{MR746961}.\medskip
\end{proof}

\begin{proof}
(roughly follows the method of \cite[Thm. 5.13]{MR0338129}) Write $X=%
\limfunc{Spec}R$ with $\left( R,\mathfrak{m}\right) $ an equicharacteristic
complete regular local domain of dimension $n$. We can assume $n\geq 1$ as
the case $n=0$ is trivial. The injectivity $M_{\ast }\left( X\right)
\hookrightarrow C^{0}\left( X,M_{\ast }\right) $ follows from the very
definition of $M_{\ast }\left( X\right) $ as a kernel, so we only need to
show that given some $\alpha =\left( \alpha _{y}\right) _{y\in X^{\left(
p+1\right) }}$ ($p\geq 0$) as in the middle term of%
\begin{equation*}
\cdots \longrightarrow \coprod_{x\in X^{\left( p\right) }}M_{\ast }\left(
\kappa \left( x\right) \right) \overset{\partial _{X}}{\longrightarrow }%
\coprod_{y\in X^{\left( p+1\right) }}M_{\ast -1}\left( \kappa \left(
y\right) \right) \overset{\partial _{X}}{\longrightarrow }\coprod_{z\in
X^{\left( p+2\right) }}M_{\ast -2}\left( \kappa \left( z\right) \right)
\end{equation*}%
such that $\partial _{X}\alpha =0$, there exists some $\beta =\left( \beta
_{x}\right) _{x\in X^{\left( p\right) }}$ with $\partial _{X}\beta =\alpha $%
. We will do this by picking a suitable closed subscheme $Y$ on which $%
\alpha $ is supported and show that the pushforward $i_{\ast }$ along $%
Y\hookrightarrow X$ is homotopic to zero:

For given $\alpha $ pick a closed subscheme $Y$ of $X$ of pure codimension
one containing the (finitely many) closed subsets $\overline{\{y\}}$ such
that $\alpha _{y}\neq 0$. As $R\ $is regular, $Y$ is cut out from $X$ by a
principal divisor, say $\omega \in R$, $Y=\limfunc{Spec}R/(\omega )$. By
Cohen's Structure Theorem there is a (non-canonical) ring isomorphism%
\begin{equation*}
R\simeq \kappa \left( \mathfrak{m}\right) [[w_{1},\ldots ,w_{n}]]\text{.}
\end{equation*}%
After possibly changing the isomorphism by Lemma \ref%
{Lemma_NagataGoodCoordinatesForWDistinguished} we can assume that $\omega $
is distinguished. Then by the Weierstra\ss\ Preparation Theorem \cite[Thm. 1
in \S 5.2.2]{MR746961} there is a unit $u\in R^{\times }$ and a Weierstra%
\ss\ polynomial $f\in \kappa \left( \mathfrak{m}\right) [[w_{1},\ldots
,w_{n-1}]][\tilde{w}_{n}]$ such that $\omega =u\cdot f$. Hence, w.l.o.g. $%
\omega =f$ as this generates the same ideal. Next, we mimick the diagram of 
\cite[\textit{proof of} Prop. 6.4.]{MR1418952}%
\begin{equation*}
\begin{xy} \xymatrix{ & & Z \ar[dr]^{\pi} \ar[dl]_g & \\ & Y \ar[rr]^i
\ar[dr] & & X \ar [dl] \\ & & A, & \\ } \end{xy}
\end{equation*}%
where in our setup

\begin{itemize}
\item $Y:=\limfunc{Spec}R/(\omega )=\limfunc{Spec}R/(f)$;

\item $X:=\limfunc{Spec}R=\limfunc{Spec}\kappa \left( \mathfrak{m}\right)
[[w_{1},\ldots ,w_{n}]]$;

\item $A:=\limfunc{Spec}\kappa \left( \mathfrak{m}\right) [[w_{1},\ldots
,w_{n-1}]]$;

\item $Z:=Y\times _{A}X$.
\end{itemize}

Here $g$, $\pi $ denote the product projections. The arrow $X\rightarrow A$
is not smooth (unlike its counterpart in \cite[\textit{proof of} Prop. 6.4.]%
{MR1418952}), but still flat of constant relative dimension one. The arrow $%
Y\rightarrow A$ is a finite morphism since by our Weierstra\ss\ Preparation
we have%
\begin{equation*}
\mathcal{O}_{Y}=\kappa \left( \mathfrak{m}\right) [[w_{1},\ldots ,w_{n-1}]][%
\tilde{w}_{n}]/\left( f\left( \tilde{w}_{n}\right) \right) \text{.}
\end{equation*}%
The arrow $i:Y\rightarrow X$ is a closed immersion. Finally, note that%
\begin{align}
\mathcal{O}_{Z}& =\kappa \left( \mathfrak{m}\right) [[w_{1},\ldots
,w_{n-1}]][\tilde{w}_{n}]/\left( f\left( \tilde{w}_{n}\right) \right) \ldots
\label{l329} \\
& \qquad \ldots \otimes _{\kappa \left( \mathfrak{m}\right) [[w_{1},\ldots
,w_{n-1}]]}\kappa \left( \mathfrak{m}\right) [[w_{1},\ldots
,w_{n-1}]][[w_{n}]]  \notag \\
& =\kappa \left( \mathfrak{m}\right) [[w_{1},\ldots ,w_{n-1},w_{n}]][\tilde{w%
}_{n}]/\left( f\left( \tilde{w}_{n}\right) \right) \text{.}  \notag
\end{align}%
We get a closed immersion $\sigma $ coming from the diagonal ideal $\left(
w_{n}-\tilde{w}_{n}\right) $ in $\mathcal{O}_{Z}$. This realizes $Y$ as the
underlying closed subscheme of $Z$ of a principal divisor. We let $V$ be the
open complement of the closed subset $\sigma \left( Y\right) $ in $Z$ so
that $Z=V\cup \sigma \left( Y\right) $ disjointly. Algebraically,%
\begin{equation*}
\mathcal{O}_{V}=\kappa \left( \mathfrak{m}\right) [[w_{1},\ldots
,w_{n-1},w_{n}]][[\tilde{w}_{n}]]/\left( f\left( \tilde{w}_{n}\right)
\right) [\tfrac{1}{w_{n}-\tilde{w}_{n}}]\text{.}
\end{equation*}%
Since $f$ is a nonzerodivisor, $\dim Z=n$, and since $w_{n}-\tilde{w}_{n}\in 
\mathfrak{m}$, the maximal ideal gets killed after inverting $w_{n}-\tilde{w}%
_{n}$.\ The resulting maximal primes correspond bijectively to those primes
of $\mathcal{O}_{Z}$ maximal with the property not to contain $w_{n}-\tilde{w%
}_{n}$; and these are all one-dimensional. Hence, $\dim V=n-1$, and $V$ is
usually \textit{not local}! Note that the open subscheme $V$ has strictly 
\textit{lower} dimension than $Z$ (this is one of the peculiarities of local
rings). Let $g^{\prime }:V\rightarrow Y$ be the restriction of the
projection $Z\rightarrow Y$ to the open subscheme $V$. This morphism is flat
of constant relative dimension zero (whereas the full $g:Z\rightarrow Y$ has
constant relative dimension one). Define%
\begin{equation*}
H:=\pi _{\ast }\circ j_{\ast }\circ \{w_{n}-\tilde{w}_{n}\}\circ g^{\prime
\ast }
\end{equation*}%
\begin{equation*}
C^{p}\left( Y,M_{\ast }\right) \rightarrow C^{p}(V,M_{\ast })\rightarrow
C^{p}(V,M_{\ast +1})\rightarrow C^{p}\left( Z,M_{\ast +1}\right) \rightarrow
C^{p}\left( X,M_{\ast +1}\right) \text{,}
\end{equation*}%
where

\begin{itemize}
\item $g^{\prime \ast }$ is the flat pullback of constant relative dimension
one;

\item $\{w_{n}-\tilde{w}_{n}\}$ denotes the left-multiplication by units on $%
C^{\bullet }$;

\item $j:V\rightarrow Z$ is the open immersion (and flat of constant
relative dimension zero) and $j_{\ast }$ is a \textit{non}-proper
pushforward along this open immersion. \textit{Beware:} As $j$ is not
proper, $\partial \circ j_{\ast }\neq j_{\ast }\circ \partial $.

\item $\pi :Z\rightarrow X$ is the product projection. It is a \textit{finite%
} morphism, one sees this by direct inspection of eq. \ref{l329} or
abstractly by base change from the finiteness of $Y\rightarrow A$.
\end{itemize}

As for the proof of \cite[Prop. 6.4]{MR1418952} one computes that $H$ is a
chain homotopy between the closed immersion pushforward $i_{\ast }$ and the
zero morphism:%
\begin{equation}
\partial _{X}\circ H+H\circ \partial _{Y}=i_{\ast }-0:C^{p}(Y,M_{\ast
})\rightarrow C^{p+1}(X,M_{\ast +1})\text{.}  \label{l7006}
\end{equation}%
This completes the proof.
\end{proof}

One can probably prove a version for all equicharacteristic regular local
rings using Panin's method via N\'{e}ron-Popescu desingularization \cite%
{MR2024050}. However, then one needs the cocontinuity property $M_{\ast }(%
\limfunc{colim}F_{i})\cong \limfunc{colim}M_{\ast }(F_{i})$ which holds for $%
K$-theory, but need \textit{not} hold for big cycle modules without adding
further axioms. There are interesting big cycle modules without this
property: Define%
\begin{equation*}
"\left. K_{n}^{\limfunc{top}}(F)\right. ":=K_{n}^{M}(F)/\left( \text{%
divisible elements}\right) \text{.}
\end{equation*}%
This is a big cycle module. The name is inspired from Fesenko's topological
Milnor $K$-groups \cite{MR1850194}. They appear in local class field theory,
but are only defined for higher local fields. The link stems from:

\begin{theorem}
\label{ThmFesenkoAlgTopKM}(Fesenko \cite[Thm. 4.7 (iv)]{MR1850194}) Let $F$
be an $n$-local field with $\limfunc{char}F=p>0$ and last residue field
algebraic over $\mathbf{F}_{p}$. Then $K_{n}^{\limfunc{top}}(F)\cong
K_{n}(F)/\left( \text{divisible elements}\right) $ as abstract groups.
\end{theorem}

It is now easy to obtain the version of eq. \ref{lX_101} for $K^{\limfunc{top%
}}$ by its counterpart for ordinary Milnor $K$-theory.

\begin{acknowledgements}
{I would like to thank P. Arndt, B. Kahn and A. Vishik for
answering some questions I had. Moreover, K. Ardakov
for spotting a mistake which then led to various simplifications. F. Trihan,
M. Morrow and I. Fesenko for their careful reading of earlier versions of this
text and very useful conversations.}
\end{acknowledgements}

\bibliographystyle{amsalpha}
\bibliography{idelobib}

\end{document}